\documentclass[onefignum,onetabnum]{siamart171218}

\usepackage{amsmath}
\usepackage{enumerate}
\usepackage{mathrsfs}
\usepackage{color}
\usepackage{tikz}
\usepackage{mathtools}
\usepackage{breqn}
\usepackage{colortbl}
\usepackage{subfigure}

\newcommand{\defword}[1]{\textcolor{blue}{\em #1}}
\newcommand{\pd}[2]{\ensuremath{ \frac{ \partial #1 }{\partial #2}  }}

\DeclareMathOperator*{\argmin}{\arg\!\min}

\newtheorem{example}[theorem]{Example}
\newtheorem{conjecture}[theorem]{Conjecture}
\newtheorem{question}[theorem]{Question}
\newtheorem{assumption}[theorem]{Assumption}
\usepackage{lipsum}
\usepackage{amsfonts}
\usepackage{graphicx}
\usepackage{epstopdf}
\usepackage{algorithmic}
\ifpdf
  \DeclareGraphicsExtensions{.eps,.pdf,.png,.jpg}
\else
  \DeclareGraphicsExtensions{.eps}
\fi

\newsiamremark{remark}{Remark}
\newsiamremark{hypothesis}{Hypothesis}
\crefname{hypothesis}{Hypothesis}{Hypotheses}
\newsiamthm{claim}{Claim}

\headers{Multistability of One-Dimensional Reaction Networks}{Xiaoxian Tang,  and Zhishuo Zhang}

\title{Multistability of Reaction Networks \\ with One-Dimensional Stoichiometric Subspaces\thanks{Submitted to the editors DATE.
\funding{XT was funded by the NSFC12001029.}
}}

\author{Xiaoxian Tang\thanks{School of Mathematical Sciences, Beihang University, Beijing,  China
  (\email{xiaoxian@buaa.edu.cn}, \url{https://sites.google.com/site/rootclassification/}).}
\and Zhishuo Zhang\thanks{School of Mathematical Sciences, Beihang University, Beijing,  China
  (\email{793008192@buaa.edu.cn}).}
}

\usepackage{amsopn}

\ifpdf
\hypersetup{
  pdftitle={Multistability of Reaction Networks with One-Dimensional Stoichiometric Subspaces},
  pdfauthor={Xiaoxian Tang,  and Zhishuo Zhang}
}
\fi

\begin{document}

\maketitle
% REQUIRED
\begin{abstract}
 For the reaction networks with one-dimensional stoichiometric subspaces,  we show the following results.
 %(1) The maximum numbers of  positive steady states and nondegenerate positive steady states are equal to each other (i.e., Nondegeneracy Conjecture is true).
 (1) If the maximum number of positive steady states is an even number $N$, then the maximum number of stable positive steady states
 is $\frac{N}{2}$. %, or $\lceil\frac{N}{2}\rceil$.
 (2) If the maximum number of positive steady states is an odd number $N$, then we provide a condition on the network such that the maximum number of stable positive steady states
 is $\frac{N-1}{2}$ if this condition is satisfied, and this maximum number is $\frac{N+1}{2}$ otherwise.
% \textcolor{blue}{4)we propose an efficient procedure for detecting multistability of given networks.}
\end{abstract}

% REQUIRED
\begin{keywords}
biochemical reaction networks, mass-action kinetics, multistationarity,
multistability, stoichiometric subspaces
\end{keywords}

% REQUIRED
\begin{AMS}
 % 68Q25, 68R10, 68U05
 92C40, 92C45
\end{AMS}

\section{Introduction}\label{sec:intro}
For the dynamical systems that arise from biochemical reaction networks, the following questions are widely open.

\begin{question}\label{question1}
%Is the maximum number of nondegenerate positive steady states equal to the maximum number of positive steady states?
Does it hold that a network admits multistability if this network admits at least three positive steady states?
\end{question}

\begin{question}\label{question2}
More generally, what is the relation between the numbers of stable positive steady states and positive steady states?
\end{question}

These questions are motivated by the multistability problem of biochemical reaction systems, which
is linked to switch-like behavior and decision-making process in cellular signaling \cite{ FM1998,BF2001, XF2003, CTF2006}.
We say a network admits multistability if there exist
positive parameters (rate constants) such that the corresponding dynamical system (arising under mass-action kinetics)
has at least two stable positive steady states in the same stoichiometric compatibility class.
Deciding the existence of multistability is a hard problem in general. So far, a typical method is first finding multistationarity (i.e., finding
parameters such that
a given network admits at least two positive steady states), and then numerically checking the stability of those steady states (e.g., \cite{OSTT2019}). Recently, symbolic methods based on computational algebraic geometry are also successfully applied to a list of biochemical reaction networks \cite{TF2020}. However, we still need simpler criterion for multistability because the standard tools (e.g., Routh-Hurwitz criterion \cite{HTX2015}, or alternatively Li\'enard-Chipart criterion \cite{datta1978}) are computationally challenging.

Since there has been a list of nice criteria for multistationarity (e.g., \cite{CF2005, ShinarFeinberg2012,CF2012, WiufFeliu_powerlaw, BP, signs, CFMW, DMST2019}), it is natural to ask what the relation between the numbers of stable positive steady states and positive steady states is. Notice that a stable steady state must be nondegenerate.  So, a related problem is the relation between the numbers of nondegenerate positive steady states and positive steady states.
For this problem, it is conjectured that these two numbers are equal if a network admits finitely many positive steady states \cite[Nondegeneracy Conjecture]{Joshi:Shiu:Multistationary}, and this conjecture is proved
for the small networks with one species, and for a sub-family of the networks with  at most two reactions (possibly reversible) \cite{Joshi:Shiu:Multistationary, shiu-dewolff}.
It is shown in numerous computational results that if a network has three nondegenerate positive steady states for a choice of parameters, then two of these steady states are stable \cite{OSTT2019,TF2020}, which suggests the answer to Question \ref{question1} might be positive. But there is no answer to the more general problem Question \ref{question2}.

In this paper, for the reaction networks with
one-dimensional stoichiometric subspaces, we answer both Question \ref{question1} Question \ref{question2} by the following three results.
\begin{itemize}
%\item[{\bf (1)}] The maximum numbers of  positive steady states and nondegenerate positive steady states are equal to each other if a network admits finitely many positive steady states (Theorem \ref{thm:nc}).
\item[{\bf (1)}]There exists a network such that its maximum number of positive steady states is $3$, and the maximum number of stable positive steady states is only $1$  (see Example \ref{ex:counter}).
So, the answer to Question \ref{question1} is negative.
\item[{\bf (2)}]
If the maximum number of positive steady states is an even number $N$, then the maximum number of stable positive steady states
 is $\frac{N}{2}$ (Theorem \ref{thm:bdm} (a)).
%If the maximum number of positive steady states is $N$ ($N<+\infty$), then the maximum number of stable positive steady states is $\lfloor\frac{N}{2}\rfloor$, or $\lceil\frac{N}{2}\rceil$ (Theorem \ref{thm:nondegmultistable}). More than that, the proof of Theorem \ref{thm:nondegmultistable} shows that if a choice of parameters yield strictly more than $3$ nondegenerate positive steady states, then at least $2$ of these steady states are stable.
   % if the maximum number of positive steady states is even and more than $3$, then the network admits multistability.
\item[{\bf (3)}]If the maximum number of positive steady states is an odd number $N$, then we provide a condition on the network such that the maximum number of stable positive steady states
 is $\frac{N-1}{2}$ if this condition is satisfied, and this maximum number is $\frac{N+1}{2}$ otherwise (Theorem \ref{thm:bdm} (b)).
%If the maximum number of positive steady states is $2N+1$ $(0<N<+\infty)$, and if a choice of parameters yields $2N+1$ ($N>1$) nondegenerate positive steady states, we prove a sufficient and necessary condition for $N+1$ of these steady states to be stable (Theorem \ref{thm:bdm}).
\end{itemize}
From the proof of Theorem \ref{thm:bdm}, we also see that if a choice of parameters yields the maximum number of positive steady states, then it also yields the maximum number of stable positive steady states (Corollary \ref{cry:bdm}). Besides the main results above, based on the lemmas for proving Theorem \ref{thm:bdm},  we provide a self-contained proof for the nondegeneracy conjecture (Theorem \ref{thm:nc}). We also provide Corollary \ref{cry:bd} for the networks with two  reactions, which shows that multistability can be easily read off from the reaction coefficients.

The idea of studying networks with one-dimensional stoichiometric subspaces is inspired by the fact that the small networks with two reactions (possibly reversible) studied in \cite{Joshi:Shiu:Multistationary,shiu-dewolff,tx2020} have one-dimensional stoichiometric subspaces. Since it is easier to eliminate variables by using the conservation-law equations, these networks have a list of nice properties as follows.  Their (nondegenerate) steady states are (simple) solutions to univariate polynomials (see Section \ref{sec:ss} and Lemma \ref{lm:nonde}).
The stability of a steady state can be determined by checking the trace of a Jacobian matrix (see Lemma \ref{lm:stable}).
By these properties, one can first prove the nondegeneracy conjecture (Theorem \ref{thm:nc}) by
perturbing a multiple solution of a univariate polynomial (see Lemma \ref{lm:even} and Lemma \ref{lm:odd}).
It is worth mentioning that some similar but different ideas of perturbing parameters can be found in the study of (nondegenerate) multistationarity
\cite{shiu-dewolff,DMST2019}.  Once the nondegeneracy conjecture is proved, it is not difficult to deduce the first part of the main result (Theorem \ref{thm:bdm} (a))  based on a ``sign condition" (see Theorem \ref{thm:sign}). The challenging part of the main result (Theorem \ref{thm:bdm} (b)) is proved by computing the Brouwer degree of a univariate polynomial \eqref{eq:q} on a bounded interval. It is remarkable that  the Brouwer degree of steady-state equations is already known  for the dissipative networks admitting no boundary steady states \cite[Theorem 3]{CFMW}. However, for a large class of dissipative (or, conservative) networks (e.g., networks with exactly two irreversible reactions), if they admit multistability, then they definitely admit boundary steady states (see \cite[Theorem 4.8]{tx2020}). So, the results in \cite{CFMW} do not apply here.

The rest of this paper is organized as follows.
In Section~\ref{sec:back}, we review the definitions of multistationarity and multistability for the
mass-action kinetics systems  arising  from reaction networks.  In Section \ref{sec:ps}, we formally present the main result: %Theorem \ref{thm:nc}, Theorem \ref{thm:nondegmultistable}, and
Theorem \ref{thm:bdm}. We also present two useful corollaries, and we illustrate how these results work by two examples.
In Section~\ref{sec:unipoly}, we provide two useful lemmas on univariate polynomials.
In Section~\ref{sec:dim1}, we prepare a list of nice properties for the networks with one-dimensional stoichiometric subspaces.
In Section~\ref{sec:nc}, we prove the nondegeneracy conjecture for the networks with one-dimensional stoichiometric subspaces (Theorem \ref{thm:nc}).
% and Theorem \ref{thm:nondegmultistable}.
In Section~\ref{sec:mstab}, we prove Theorem \ref{thm:bdm} and its corollaries.
Finally, we end this paper with some future directions inspired by Theorem \ref{thm:bdm}, see Section \ref{sec:dis}.
\section{Background}\label{sec:back}
\subsection{Chemical reaction networks}\label{sec:pre}

%\subsection{Reaction networks}\label{sec:reaction}
%{\color{red} Put a more general definition for CRNs here.  Define mass-action ODEs. Steady states. Stability.}

%\textcolor{red}{check the notation $N$ and $S$.--Xiaoxian}

In this section, we briefly recall the standard notions and definitions on reaction networks, see \cite{CFMW, Joshi:Shiu:Multistationary} for more details.
%As in~\cite{DPST}, our notation closely matches that of Conradi, Feliu, Mincheva, and Wiuf~\cite{CFMW}.
A \defword{reaction network} $G$  (or \defword{network} for short) consists of a set of $s$ species $\{X_1, X_2, \ldots, X_s\}$ and a set of $m$ reactions:
\begin{align}\label{eq:network}
\alpha_{1j}X_1 +
 \dots +
\alpha_{sj}X_s
%~ \xrightarrow{\kappa_j} ~
~ \xrightarrow{} ~
\beta_{1j}X_1 +
 \dots +
\beta_{sj}X_s,
 \;
    {\rm for}~
	j=1,2, \ldots, m,
\end{align}
where all $\alpha_{ij}$ and $\beta_{ij}$ are non-negative integers, and
$(\alpha_{1j},\ldots,\alpha_{sj})\neq (\beta_{1j},\ldots,\beta_{sj})$. We call the $s\times m$ matrix with
$(i, j)$-entry equal to $\beta_{ij}-\alpha_{ij}$ the
\defword{stoichiometric matrix} of
$G$,
denoted by ${\mathcal N}$.
%Let $d=s-{\rm rank}(N)$. %, and
We call the image of ${\mathcal N}$
the \defword{stoichiometric subspace}, denoted by $S$.
%let $\mathrm{im}(N)^{\perp}$ denote the orthogonal complement of the image of $N$.
%$S^{\perp}$.
%$\mathrm{im}(N)^{\perp}$.

We denote by $x_1, x_2, \ldots, x_s$ the concentrations of the species $X_1,X_2, \ldots, X_s$, respectively.
Under the assumption of mass-action kinetics, we describe how these concentrations change  in time by following system of ODEs:
\begin{align}\label{eq:sys}
\dot{x}~=~f(\kappa; x)~=~(f_1(\kappa; x), \ldots, f_s(\kappa; x))^{\top}~:=~{\mathcal N}\cdot \begin{pmatrix}
\kappa_1 \, x_1^{\alpha_{11}}
		x_2^{\alpha_{21}}
		\cdots x_s^{\alpha_{s1}} \\
\kappa_2 \, x_1^{\alpha_{12}}
		x_2^{\alpha_{22}}
		\cdots x_s^{\alpha_{s2}} \\
		\vdots \\
\kappa_m \, x_1^{\alpha_{1m}}
		x_2^{\alpha_{2m}}
		\cdots x_s^{\alpha_{sm}} \\
\end{pmatrix}~,
\end{align}
where $x$ denotes the vector $(x_1, x_2, \ldots, x_s)$,
and each $\kappa_j \in \mathbb R_{>0}$ is called a \defword{rate constant} corresponding to the
$j$-th reaction in \eqref{eq:network}.
 By considering the rate constants as an unknown vector $\kappa:=(\kappa_1, \kappa_2, \dots, \kappa_m)$, we have polynomials $f_{i}(\kappa;x) \in \mathbb Q[\kappa, x]$, for $i=1,2, \dots, s$.

 Let $d:=s-{\rm rank}({\mathcal N})$. A \defword{conservation-law matrix} of $G$, denoted by $W$, is any row-reduced $d\times s$-matrix whose rows form a basis of $S^{\perp}$ (note here, ${\rm rank}(W)=d$). Our system~\eqref{eq:sys} satisfies $W \dot x =0$,  and both the positive orthant $\mathbb R_{>0}^s$ and its closure $\mathbb R_{\ge 0}$ are forward-invariant for the dynamics. Thus,
any trajectory $x(t)$ beginning at a non-negative vector $x(0)=x^0 \in
\mathbb{R}^s_{> 0}$ remains, for all positive time,
 in the following \defword{stoichiometric compatibility class} with respect to the  \defword{total-constant vector} $c:= W x^0 \in {\mathbb R}^d$:
\begin{align*}%\label{eq:pc}
{\mathcal P}_c~:=~ \{x\in {\mathbb R}_{\geq 0}^s \mid Wx=c\}.~
\end{align*}
That means ${\mathcal P}_c$ is forward-invariant with
respect to the dynamics~\eqref{eq:sys}.

\subsection{Multistationarity and Multistability}\label{sec:mm}
A \defword{steady state} %(or, simply \defword{steady state})\footnote{Usually, a steady state is defined as a non-negative vector $x\in {\mathbb R}_{\geq 0}^s$.
%In our setting, we do not
%consider boundary steady states (i.e., steady states with zero coordinates). So all steady states in our context are positive.}
of~\eqref{eq:sys} is a concentration vector
$x^* \in \mathbb{R}_{\geq 0}^s$ such that $f(x^*)=0$, where  $f(x)$ is on the
right-hand side of the
ODEs~\eqref{eq:sys}.
If all coordinates of a steady state $x^*$ are strictly positive (i.e., $x^*\in \mathbb{R}_{> 0}^s$), then we call $x^*$ a \defword{positive steady state}.
If a steady state $x^*$ has zero coordinates (i.e., $x^*\in \mathbb{R}_{\geq 0}^s\backslash \mathbb{R}_{> 0}^s$), then we call $x^*$ a \defword{boundary steady state}.
%A steady state $x^*$ is {\em nondegenerate} if ${\rm Im}\left( df_{\kappa} (x^*)|_{S} \right) = \St$.
%(Here, $df_{\kappa}(x^*)$ is the Jacobian matrix of $f_{\kappa}$ at $x^*$.)
%A nondegenerate steady state is
%{\em hyperbolic} if each of the $\sigma:= \dim (\St)$ nonzero eigenvalues of $df_{\kappa}(x^*)$ has nonzero real part and is
%{\em exponentially stable} if each of the $\sigma:= \dim(\St)$ nonzero eigenvalues of $df_{\kappa}(x^*)$ has negative real part.
%We say a steady state $x^*$ is a \defword{positive steady states} if $x ^* \in \mathbb{R}^s_{> 0}$. % and \defword{boundary steady states}
%$x^*\in {\mathbb R}_{\geq 0}^s\backslash {\mathbb R}_{>0}^s$.
We say a steady state $x^*$ is \defword{nondegenerate} if
${\rm im}\left({\rm Jac}_f (x^*)|_{S}\right)=S$,
% = \St$.
where ${\rm Jac}_f(x^*)$ denotes the Jacobian matrix of $f$, with respect to $x$, at $x^*$.
%A nondegenerate steady state $x^*$ is \defword{Liapunov stable} if for any $\epsilon>0$ and for any $t_0>0$, there exists $\delta>0$ such that
%$\parallel x(t_0)-x^*\parallel<\delta$ implies $\parallel x(t)-x^*\parallel<\epsilon$ for any $t\geq t_0$.
%A Liapunov stable steady state $x^*$ is  \defword{locally asymptotically stable}  if  there exists $\delta>0$  such that
%$\parallel x(t_0)-x^*\parallel<\delta$ implies $\lim_{t\rightarrow \infty} x(t)=x^*$.
A steady state $x^*$ is \defword{exponentially stable} (or, simply \defword{stable} in this paper)
if it is nondegenerate, and  all non-zero eigenvalues of ${\rm Jac}_f(x^*)$ have negative real parts.
%if all non-zero eigenvalues of ${\rm Jac}_f(x^*)$ have negative real parts.
Note that if a steady state is exponentially stable, then it is locally asymptotically stable \cite{P2001}.
%Notice that when the stoichiometric matrix $N$ is full rank, a steady state $x^*$ is nondegenerate if
%  ${\rm Jac}_f(x^*)$ is full rank.

%A nondegenerate steady state is
%{\em hyperbolic} if each of the $\sigma:= \dim (\St)$ nonzero eigenvalues of $df_{\kappa}(x^*)$ has nonzero real part and is
%\defword{exponentially stable} if every nonzero eigenvalue of ${\rm Jac}_f(x^*)$ has a negative real part.

%A nondegenerate steady state $x^*$ is said to be \defword{Liapunov stable} if for any $\epsilon>0$ and for any $t_0>0$, there exists $\delta>0$ such that
%$\parallel x(t_0)-x^*\parallel<\delta$ implies $\parallel x(t)-x^*\parallel<\epsilon$ for any $t\geq t_0$.
%A Liapunov stable steady state $x^*$ is said to be \defword{locally asymptotically stable}  if  there exists $\delta>0$  such that
%$\parallel x(t_0)-x^*\parallel<\delta$ implies $\lim_{t\rightarrow \infty} x(t)=x^*$.

%It is well-known that if a nondegenerate steady state $x^*$ is exponentially stable, then
%it is locally asymptotically stable.

Suppose $N\in {\mathbb Z}_{\geq 0}$. We say a  network  \defword{admits $N$ (nondegenerate) positive steady states}  if for some rate-constant
vector $\kappa$ and for some total-constant vector $c$,  it has $N$ (nondegenerate) positive steady states  in the stoichiometric compatibility class ${\mathcal P}_c$.
%A  network \defword{admits $N$ nondegenerate positive steady states}  if for some rate-constant
%vector $\kappa$ and for some total-constant vector $c$,  it has $N$ nondegenerate positive steady states  in the same stoichiometric compatibility class ${\mathcal P}_c$.
Similarly, we say a  network  \defword{admits $N$ stable positive steady states}  if for some rate-constant
vector $\kappa$ and for some total-constant vector $c$,  it has $N$ stable positive steady states  in ${\mathcal P}_c$.

The \defword{maximum number of positive steady states} of a network $G$ is
{\footnotesize
\[cap_{pos}(G)\;:=\;\max\{N\in {\mathbb Z}_{\geq 0}\cup \{+\infty\}|G \;\text{admits}\; N\; \text{positive steady states}\}.\]
}
The \defword{maximum number of nondegenerate positive steady states} of a network $G$ is
{\footnotesize
\[cap_{nondeg}(G)\;:=\;\max\{N\in {\mathbb Z}_{\geq 0}\cup \{+\infty\}|G \;\text{admits}\; N\; \text{nondegenerate positive steady states}\}.\]
}
The \defword{maximum number of stable positive steady states} of a network $G$ is
{\footnotesize
\[cap_{stab}(G)\;:=\;\max\{N\in {\mathbb Z}_{\geq 0}\cup \{+\infty\}|G \;\text{admits}\; N\; \text{stable positive steady states}\}.\]
}
%It is obvious that if $\hat{G}$ has the form of $G$,
%then $cap_{pos}(\hat{G})=cap_{pos}(G)$, $cap_{nondeg}(\hat{G})=cap_{nondeg}(G)$, and
 %$cap_{stab}(\hat{G})=cap_{stab}(G)$.
We say a network admits \defword{multistationarity} if $cap_{pos}(G)\geq 2$.
%We say a network admits \defword{nondegenerate multistationarity} if $cap_{nondeg}(G)\geq 2$.
We say a network admits \defword{multistability} if $cap_{stab}(G)\geq 2$.
\begin{conjecture}[Nondegeneracy Conjecture]\cite[Conjecture 2.3]{Joshi:Shiu:Multistationary}\label{conj:nc}
%\textcolor{red}{nondegenerate conjecutre}
%Suppose $G\in {\mathcal G}_0$, or, suppose $G\in {\mathcal G}_1\cup {\mathcal G}_2$ and $G$ has up to $2$ species.
Given a network $G$,
if $cap_{pos}(G)<+\infty$, then $cap_{nondeg}(G)=cap_{pos}(G)$.
\end{conjecture}

%\subsection{Nondegeneracy conjecture}\label{sec:nondeg}
%it is not necessary to highlight the conjecture. otherwise, people think this work is about this conjecture. But it is not true.

\section{Main Result}\label{sec:ps}
In this section,   we first define a list of notions needed in the statement of the main result, and after that, we formally present the main result: %Theorem \ref{thm:nc}, Theorem \ref{thm:nondegmultistable}, and
Theorem \ref{thm:bdm}, which reveals the relation between the maximum numbers of stable positive steady states and positive steady states. We also provide two useful corollaries.
Corollary \ref{cry:bdm} shows that if we want to
find the parameters such that a network (with a one-dimensional stoichiometric subspace) has the maximum number of stable positive steady states, then one method is to search the parameters such that the network has the maximum number of positive steady states.
Corollary \ref{cry:bd} shows that if a network contains two reactions, then we can determine if it admits multistability by looking at the reaction coefficients.
Theorem \ref{thm:bdm} and the two corollaries will be proved in Section \ref{sec:mstab}. We illustrate how these results work by two examples, where Example \ref{ex:counter} answers Question \ref{question1}. We remark that Theorem \ref{thm:bdm} generalizes \cite[Theorem 3.16]{tx2020}.

\begin{assumption}\label{assumption}
For any network $G$ with reactions defined in \eqref{eq:network},
by the definition of reaction network,
we know $(\alpha_{11},\ldots, \alpha_{s1})\neq (\beta_{11},\ldots, \beta_{s1})$.
 Without loss of generality, we will assume %there exists $\lambda \in {\mathcal R}\backslash \{\}$
  $\beta_{11}-\alpha_{11}\neq 0$ throughout this paper.
\end{assumption}

Suppose a network $G$ \eqref{eq:network} (with $s$ species and $m$ reactions) has a one-dimensional stoichiometric subspace. Note that the rank of
its conservation-law matrix is $s-1$.
Given a total-constant vector $c^*\in {\mathbb R}^{s-1}$, under Assumption \ref{assumption}, %for every $i\in \{1, \ldots, s\}$,
we define
\begin{align}
&  A_1 \;:=\; 1, \;\;\;\;\;\;\;\;\;\;\;\;\;\;\;\; B_1 \;:=\; 0, \label{eq:a}\\
& A_i \;:=\; \frac{\beta_{i1}-\alpha_{i1}}{\beta_{11}-\alpha_{11}}, \;\;\; B_i \;:=\; -\frac{c^*_{i-1}}{\beta_{11}-\alpha_{11}}, \;\; \text{for any}\; i\in \{2, \ldots, s\}, \label{eq:b}
 \end{align}
and for every $i\in \{1, \ldots, s\}$,  we define an equivalence class
\begin{align}\label{eq:iec}
[i]_{c^*} \;:=\;
\begin{cases}
\{k\in \{1, \ldots, s\}|A_k\neq 0, \;\text{and} \; \frac{B_k}{A_k}=\frac{B_i}{A_i} \}, &\text{if}\; A_i\neq 0, \\
    \{k\in \{1, \ldots, s\}|A_k=0 \},   &\text{if}\;  A_i=0.
\end{cases}
\end{align}
%we define a map
%\begin{align}\label{eq:varphi}
%\varphi\;:\;\{1, \ldots, s\}\times \{1, \ldots, m\} \;\rightarrow\; {\mathbb N}_{\geq 0}
%\end{align}
%as follows.
%For any $i\in \{1, \ldots, s\}$, if
%$A_i=0$, then for any
%$j\in\{1, \ldots, m\}$, we define $\varphi(i,j):=0$, and if
%$A_i\neq 0$, then $\varphi(i,j)$ is the maximum integer such that
%$$(A_ix_1+B_i)^{\varphi(i, j)}\; \text{divides}\;
%\prod\limits_{k=1}^s(A_kx_1+B_k)^{\alpha_{kj}} \;\text{in}\; {\mathbb R}[x_1],$$
%where
 When $c^*$ is clear from the context, we simply denote $[i]_{c^*}$ by $[i]$.
Denote by $r$ the number of these equivalence classes, i.e.,
\begin{align}\label{eq:r}
r\; :=\; |\;\{\;[i]\;|\;i\in \{1, \ldots, s\}\;\}\;|\;\leq\; s.
\end{align}
    Without loss of generality, we assume $1, \ldots, r$ are the representatives of these equivalence classes.
    Recall $m$ is the number of reactions in $G$ \eqref{eq:network}.
    For every $k\in \{1, \ldots, r\}$,  we define
    \begin{align}\label{eq:varphi}
    \varphi_k \;&:=\; \min \limits_{1\leq j\leq m} \{\sum_{i\in[k]}\alpha_{ij}\},\;
    \text{where}\; \alpha_{ij}\text{'s are the reaction coefficients, see \eqref{eq:network}}.
   % \gamma_{kj}\;& :=\; \sum_{i\in[k]}\alpha_{ij}-\varphi_k. \label{eq:gamma}
    \end{align}
   % $$\mu_{kj}\;:=\; \sum_{i\in[k]}\alpha_{ij},$$
    And for every $j\in \{1, \ldots, m\}$, %for every $k\in \{1, \ldots, r\}$,
    we define
    \begin{align}\label{eq:gamma}
    %\varphi_k \;&:=\; \min \limits_{1\leq j\leq m} \{\sum_{i\in[k]}\alpha_{ij}\}, \label{eq:varphi} \\
    \gamma_{kj}\;& :=\; \sum_{i\in[k]}\alpha_{ij}-\varphi_k.
    \end{align}
Define two sets of indices
 \begin{align}
 %{\mathcal H} \;:=\; \{i\in \{1,\ldots, s\}|\varphi(i,1), \ldots, \varphi(i,m)\;\text{are not all the same}\}.
 %{\mathcal J}_0 & \;:=\; \{i\in \{1,\ldots, s\}|A_i=0\}, \label{eq:nindexj0}\\
{\mathcal J} & \;:=\; \{i\in \{1,\ldots, s\}|A_i\neq 0\}, \;\;\;\;\;\;\; \text{and} \label{eq:nindexj1}\\
 {\mathcal H} &\;:=\; \{k\in \{1,\ldots, r\}\cap {\mathcal J}|\gamma_{k1}, \ldots, \gamma_{km}\;\text{are not all the same}\}. \label{eq:nindexi}
 \end{align}
 \begin{assumption}\label{assumption2}
Given a network $G$,
   suppose the stoichiometric subspace of $G$ is one-dimensional.
   For a total-constant vector $c^*$, let ${\mathcal H}$ be the set of indices defined in \eqref{eq:nindexi}.
   If there exists a  rate-constant vector $\kappa^*$ such that $G$ has $N$ $(0<N<+\infty)$ positive steady states in the stoichiometric compatibility class ${\mathcal P}_{c^*}$,  then
 without loss of generality, we will assume %there exists $\lambda \in {\mathcal R}\backslash \{\}$
  $1\in {\mathcal H}$.
\end{assumption}
We will explain later in Section \ref{sec:dim1} why Assumption \ref{assumption2} does not lose generality. See
Lemma \ref{lm:assumption2}, Remark \ref{rmk:assumption2} and Example \ref{ex:assumption2}.

For every $i\in \{1, \ldots, s\}$, define an open interval
 \begin{align}\label{eq:iinterval}
I_i \;:= \;
\begin{cases}
(-\frac{B_i}{A_i}, +\infty)&\text{if}\; A_i>0\\
(0, +\infty)&\text{if}\; A_i=0\\
(0, -\frac{B_i}{A_i})& \text{if}\;  A_i<0
\end{cases}.
\end{align}
Under Assumption \ref{assumption2}, define an open interval
 \begin{align}\label{eq:tinterval}
I \;:=\; \bigcap\limits_{k\in {\mathcal H}} I_k, \;\;\; \text{where}\;
I_k \; \text{is defined in \eqref{eq:iinterval}.}
\end{align}
%Suppose the left endpoint of the above interval $I$ is $a$.
\begin{remark}\label{rmk:end1}
Under Assumption \ref{assumption2}, we have $1\in {\mathcal H}$. So, by \eqref{eq:tinterval},
$I \subset I_1=(0, +\infty)$. Therefore, the left endpoint of $I$ \eqref{eq:tinterval} is no less than $0$.
 \end{remark}
Let
\begin{align}\label{eq:setk}
% {\mathcal K} \;:=\;
 %\begin{cases}
 %\{k\in {\mathcal H}|A_k>0 \;\text{and} \;-\frac{B_k}{A_k} \;=\;a
 %\max \limits_{i\in {\mathcal H}} \{-\frac{B_i}{A_i}\}
 %\}, & a>0 \\
 %[1], & a=0.
 %\end{cases}
  \tau\;\text{be the index in}\;{\mathcal H} \;\text{such that}\; A_{\tau}>0 \;\text{and} \;-\frac{B_{\tau}}{A_{\tau}} \;\text{is the left endpoint of}\;I.
 %\max \limits_{i\in {\mathcal H}} \{-\frac{B_i}{A_i}\}
 \end{align}
 Note that by \eqref{eq:varphi} and \eqref{eq:gamma}, there exists $j\in \{1, \ldots, m\}$ such that
 $\gamma_{\tau j}=0$.
 We define
 \begin{align}\label{eq:L}
{\mathcal L} \;:=\; \{j\in \{1, \ldots, m\}| \gamma_{\tau j}\;=\; 0 \}.
%& {\mathcal K} \;\text{and}\; {\mathcal L} \;\text{are defined in}\; \text{\textcolor{red}{(5.10)}} \;\text{and}\;\text{\textcolor{red}{(5.11)}}.
 \end{align}

Now, we are prepared to define two important sets of parameters, which are depending on the given network $G$:
{\footnotesize
\begin{align}
{\mathcal W} \;&:=\; \{(\kappa^*, c^*)|\text{for the rate-constant vector}\; \kappa^*,  G\; \text{has}\;cap_{pos}(G)\; \text{positive steady states in}\; {\mathcal P}_{c^*} \},  \label{eq:witness}\\
\text{and} & \notag\\
{\mathcal B} \;&:=\; \{(\kappa^*, c^*)|\sum\limits_{j\in \mathcal L}(\beta_{1j}-\alpha_{1j})\kappa^*_{j}\prod\limits_{i\in {\mathcal J}}|A_i|^{\alpha_{ij}}\prod\limits_{i\in \{1, \ldots, s\}\backslash{\mathcal J}}B_i^{\alpha_{ij}} \prod \limits_{k\in {\mathcal H}, k\neq \tau}(\frac{B_k}{|A_k|}-\frac{A_k}{|A_k|}\frac{B_{\tau}}{A_{\tau}})^{\gamma_{kj}}>0\} \label{eq:bwd}
\end{align}
}(here, we remark that in \eqref{eq:bwd}, the values of $A_i$, $B_i$, $\gamma_{kj}$, ${\mathcal J}$, ${\mathcal H}$, $\tau$, and ${\mathcal L}$ are depending on $c^*$ (recall \eqref{eq:a}--\eqref{eq:L})).
We explain the meaning of the above two sets as follows.
\begin{itemize}
\item[(i)]The set ${\mathcal W}$ collects all ``witnesses" $(\kappa^*, c^*)$ such that the number of positive steady states  reaches its maximum.
\item[(ii)]The set ${\mathcal B}$ collects all choices of parameters $(\kappa^*, c^*)$ such that
the ``Brouwer degree" is positive (see more discussion in Section \ref{sec:dis}, and also see the meaning of Brouwer degree in Remark \ref{rmk:nobd}, or in
\cite[Theorem 1]{CFMW}).
\end{itemize}

\begin{theorem}\label{thm:bdm}
Given a network $G$ \eqref{eq:network} with a one-dimensional stoichiometric subspace, suppose $cap_{pos}(G)<+\infty$.
 \begin{itemize}
\item[(a)] If $cap_{pos}(G)$ is even, then $cap_{stab}(G)\;=\;\frac{cap_{pos}(G)}{2}$.
\item[(b)] If $cap_{pos}(G)$ is odd, then
\begin{align*}
cap_{stab}(G)\;=\;
\begin{cases}
\frac{cap_{pos}(G)-1}{2} & \text{if}\;\;\; {\mathcal W}  \cap {\mathcal B}=\emptyset  \\
\frac{cap_{pos}(G)+1}{2}  & \text{if}\;\;\;  {\mathcal W}   \cap {\mathcal B}\neq \emptyset
\end{cases},
\end{align*}
 \end{itemize}
 where ${\mathcal W}$ and ${\mathcal B}$ are defined in \eqref{eq:witness}--\eqref{eq:bwd}.
\end{theorem}

%\begin{theorem}\label{thm:bdm}
%Given a network $G$ \eqref{eq:network}, suppose the stoichiometric subspace of $G$ is one-dimensional.
%For a total-constant vector $c^*$,
%let
%${\mathcal H}$ be the set of indices defined in \eqref{eq:nindexi}.
% If $cap_{pos}(G)=2N+1$ $(0<N<+\infty)$, and if
% for a rate-constant vector $\kappa^*$, $G$ has $2N+1$ nondegenerate positive
% steady states in the stoichiometric compatibility class ${\mathcal P}_{c^*}$, then $N+1$ of these steady states are stable if and only if
% \begin{align}\label{eq:bwd}
% \sum\limits_{j\in \mathcal L}(\beta_{1j}-\alpha_{1j})\kappa^*_{j}\prod\limits_{i\in {\mathcal J}}|A_i|^{\alpha_{ij}}\prod\limits_{i\in \{1, \ldots, s\}\backslash{\mathcal J}}B_i^{\alpha_{ij}} \prod \limits_{k\in {\mathcal H}, k\neq \tau}(\frac{B_k}{|A_k|}-\frac{A_k}{|A_k|}\frac{B_{\tau}}{A_{\tau}})^{\gamma_{kj}}>0.
% \end{align}
% where for the total constant-vector $c^*$, the notions $A_i$, $B_i$, $\gamma_{kj}$, ${\mathcal J}$, ${\mathcal H}$, $\tau$, and ${\mathcal L}$ are defined in \eqref{eq:a}--\eqref{eq:L}.
 %where
%  \begin{align}
% A_j \;\text{and}\; B_j \; \text{are defined in} \; \eqref{eq:a} \;\text{and}\; \eqref{eq:b},
% \end{align}
 %\begin{align}\label{eq:mathcalC}
 %{\mathcal C}_j \;=\; \prod\limits_{i\in {\mathcal J}}|A_i|^{\alpha_{ij}}\prod\limits_{i\in \{1, \ldots, s\}\backslash{\mathcal J}}B_i^{\alpha_{ij}}.
 %\end{align}
% $$\tau \;\text{is the index defined in \eqref{eq:setk}, and}$$
%\end{theorem}

\begin{corollary}\label{cry:bdm}
Given a network $G$ \eqref{eq:network} with a one-dimensional stoichiometric subspace, suppose $cap_{pos}(G)<+\infty$.
Then, there exist a rate-constant
vector $\kappa$ and a total-constant vector $c$ such that the network $G$ has $cap_{pos}(G)$ positive steady states in the stoichiometric compatibility class ${\mathcal P}_c$, where
$cap_{stab}(G)$ of these positive steady states are stable.
\end{corollary}

\begin{corollary}\label{cry:bd}
 Given a network $G$ \eqref{eq:network}, suppose $G$ contains two reactions and $cap_{pos}(G)<+\infty$.
 \begin{itemize}
\item[(a)] If $cap_{pos}(G)$ is even, then $cap_{stab}(G)\;=\;\frac{cap_{pos}(G)}{2}$.
\item[(b)] If $cap_{pos}(G)$ is odd, then
\begin{align*}
cap_{stab}(G)\;=\;
\begin{cases}
\frac{cap_{pos}(G)-1}{2} & \text{if}\;\;\; {\mathcal W}  \cap {\mathcal B}_{2-react}=\emptyset  \\
\frac{cap_{pos}(G)+1}{2}  & \text{if}\;\;\;  {\mathcal W}   \cap {\mathcal B}_{2-react}\neq \emptyset
\end{cases},
\end{align*}
 \end{itemize}
 where ${\mathcal W}$  is defined in \eqref{eq:witness}, and
 \begin{align}\label{eq:bd}
 {\mathcal B}_{2-react} \;&:=\; \{(\kappa^*, c^*)|(\gamma_{\tau2}-\gamma_{\tau1})(\beta_{\tau1}-\alpha_{\tau1})>0\}
 \end{align}
 (here, recall $\tau$ \eqref{eq:setk} and $\gamma_{\tau j}$ $(j\in\{1,2\})$ \eqref{eq:gamma} are depending on the total constant-vector $c^*$). Especially, if $\tau$ is the only index in $[\tau]$ (recall \eqref{eq:iec}), then
 the inequality stated in \eqref{eq:bd} can be replaced with
 \begin{align}\label{eq:sbd}
 (\alpha_{\tau2}-\alpha_{\tau1})(\beta_{\tau1}-\alpha_{\tau1})>0.
 \end{align}
 \end{corollary}

% \begin{corollary}\label{cry:bd}
% Given a network $G$ \eqref{eq:network}, suppose $G$ contains two reactions.
%  If $cap_{pos}(G)=2N+1$ $(0<N<+\infty)$, and if
% for a rate-constant vector $\kappa^*$, $G$ has $2N+1$ nondegenerate positive
% steady states in a stoichiometric compatibility class ${\mathcal P}_{c^*}$, then $N+1$ of these steady states are stable if and only if
 %If $cap_{pos}(G)=3$, and if
 %for a rate-constant vector $\kappa^*$ and a total-constant vector $c^*$, $G$ has three nondegenerate positive
 %steady states, then two of these steady states are stable if and only if
% \begin{align}\label{eq:bd}
% (\gamma_{\tau1}-\gamma_{\tau2})(\beta_{\tau1}-\alpha_{\tau1})<0,
% \end{align}
% where $\tau$ and $\gamma_{\tau j}$ are  defined in \eqref{eq:setk} and \eqref{eq:gamma} for the total constant-vector $c^*$. Especially, if $\tau$ is the only index in $[\tau]$, then
% the condition \eqref{eq:bd} becomes
% \begin{align}\label{eq:sbd}
% (\alpha_{\tau1}-\alpha_{\tau2})(\beta_{\tau1}-\alpha_{\tau1})<0.
% \end{align}
% \end{corollary}

 \begin{example}\label{ex:main1}
 This example illustrates how to compute $cap_{stab}(G)$ using Theorem \ref{thm:bdm}.
  Consider the following network $G$
 $$ X_2\xrightarrow{\kappa_1} X_1,\;\;\;
 X_1\xrightarrow{\kappa_2} X_2,\;\;\; 2X_1+X_2\xrightarrow{\kappa_3} \textcolor{black}{3}X_1. $$
 %It is easy to check the following results in this network.
By \eqref{eq:network}, for this network, we have $s=2$.
 %$m=3$,  $\beta_{11}-\alpha_{11}=1$ and $\beta_{21}-\alpha_{21}=-1$.
Note that there is only one total constant $c_1$ since $s-1=1$.
By \cite[Theorem 2.6]{tx2020}, we know $cap_{pos}(G)=3$. The key idea of computing $cap_{stab}(G)$ is to check if we have $ {\mathcal W}\cap {\mathcal B}\neq \emptyset$
for the two sets ${\mathcal W}$ \eqref{eq:witness} and ${\mathcal B}$ \eqref{eq:bwd}.
In fact, we have
\begin{align}\label{eq:main1}
(\kappa_1^*=\frac{1}{2},\kappa_2^*=16, \kappa_3^*=\textcolor{black}{\frac{3}{2}}, c_1^*=-9)\in {\mathcal W}\cap {\mathcal B}
\end{align}
(we list the details in the next paragraph).
Therefore, by Theorem \ref{thm:bdm} (b), we have $cap_{stab}(G)=\frac{cap_{pos}(G)+1}{2}=2$.

Below, we show that \eqref{eq:main1} is true. By \eqref{eq:network}, for this network, we have  $\beta_{11}-\alpha_{11}=1$ and $\beta_{21}-\alpha_{21}=-1$.
So, for the total constant $c_1^*=-9$ ($\in \mathbb{R}^1$), by \eqref{eq:a}, we have $A_1=1$ and $B_1=0$, and by \eqref{eq:b},
we have $A_2=-1$ and $B_2=9$. Hence, it is straightforward to verify that there are two equivalence classes defined in \eqref{eq:iec}:
$$[1]=\{1\},\;\text{and}\;\; [2]=\{2\},$$
and the set of indices defined in \eqref{eq:nindexj1} is ${\mathcal J}=\{1,2\}$.
By \eqref{eq:varphi}, for each equivalence class, we have
 $$\varphi_1=\min\limits_{1\leq j\leq 3} \{\alpha_{1j}\}=0,\;\text{and}\;\varphi_2=\min\limits_{1\leq j\leq 3} \{\alpha_{2j}\}=0,$$
 and so,  by \eqref{eq:gamma}, we have $\gamma_{ij}=\alpha_{ij}$ for every $i\in \{1, 2\}$ and for every $j\in \{1, 2, 3\}$.
 %$$(\gamma_{ij})=(\alpha_{ij})=
  %  \begin{pmatrix}
   %   0&1&2\\1&0&1
   % \end{pmatrix}$$
Hence, the set of indices defined in \eqref{eq:nindexi} is ${\mathcal H}=\{1,2\}$, and the open interval defined in \eqref{eq:tinterval} is $I=I_1\cap I_2=(0,9)$, where
% ${\mathcal J}=\{1,2\},{\mathcal H}=\{1,2\}.$
$I_1=(0,+\infty)$ and $I_2=(0,9)$.
Note that the left endpoint of $I$ is $0$. So,
the index $\tau$ defined in \eqref{eq:setk} is
 $\tau=1$, and the set of indices defined in \eqref{eq:L} is ${\mathcal L}=\{1\}$ (the set of indices $j$ such that $\gamma_{1j}=0$).
 So, for the rate constants $\kappa_1^*=\frac{1}{2},\kappa_2^*=16$, and $\kappa_3^*=\textcolor{black}{\frac{3}{2}}$,
 the left-hand side of the inequality in \eqref{eq:bwd} is
$$
 \sum\limits_{j\in \mathcal L}(\beta_{j1}-\alpha_{j1})\kappa^*_{j}\prod\limits_{i\in {\mathcal J}}|A_i|^{\alpha_{ij}}\prod\limits_{i\in \{1, \ldots, s\}\backslash{\mathcal J}}B_i^{\alpha_{ij}} \prod \limits_{k\in {\mathcal H}, k\neq \tau}(\frac{B_k}{|A_k|}-\frac{A_k}{|A_k|}\frac{B_{\tau}}{A_{\tau}})^{\gamma_{kj}}
$$
$$ =(\beta_{11}-\alpha_{11})\kappa_1^*\prod\limits_{i=1}^2|A_i|^{\alpha_{i1}}
\left(\frac{B_2}{|A_2|}-\frac{A_2}{|A_2|}\frac{B_1}{A_1}\right)^{\gamma_{21}}=\frac{9}{2}>0.
 $$
 That means, we have
 \begin{align*}
(\kappa_1^*=\frac{1}{2},\kappa_2^*=16, \kappa_3^*=\textcolor{black}{\frac{3}{2}}, c_1^*=-9)\in {\mathcal B}
\end{align*}
Note that by \cite[Theorem 2.6]{tx2020},  for these parameters, G has three (the maximum number) nondegenerate positive steady states. So, the above choice of parameters is also in
the set of witnesses ${\mathcal W}$.
Hence, we indeed have \eqref{eq:main1}.
 \end{example}

  \begin{example}\label{ex:counter}
  Consider the network $G$
 $$ 2X_1 + 2X_2 + X_3 \xrightarrow{\kappa_1} 3X_1 + X_2,\;\;
 X_1+2X_3 \xrightarrow{\kappa_2} X_2 + 3X_3. $$
 We show by Corollary \ref{cry:bd} that for this network, the answer to Question \ref{question1} is negative.

First, we point out that $cap_{pos}(G)=3$. In fact,
 from the total degree of the steady-state system $h$ shown later in Example \ref{ex:counter2},
 we see that the number of positive solutions of $h=0$ is upper bounded by $3$.
 And it is straightforward to check that for the rate constants $\kappa_1^*=1$, and $\kappa_2^*=\frac{8}{3}$, G has three nondegenerate positive steady states in the stoichiometric compatibility class ${\mathcal P}_{c^*=(-3,-\frac{11}{4})}$.

Next, we will prove that $cap_{stab}(G)=1$.
By Corollary \ref{cry:bd} (b), it is sufficient to show that
${\mathcal W}\cap {\mathcal B}_{2-react}=\emptyset$ (here, ${\mathcal B}_{2-react}$ is defined in \eqref{eq:bd}). %In fact, we show below that
%${\mathcal B}_{2-react}=\emptyset$.
\textcolor{black}{
By \eqref{eq:network}, for this network, we have $s=3$, $m=2$,  $\beta_{11}-\alpha_{11}=1$, $\beta_{21}-\alpha_{21}=-1$, and $\beta_{31}-\alpha_{31}=-1$.} Notice that the values of $A_i$'s defined in \eqref{eq:a}--\eqref{eq:b} have nothing to do with the choice of the total-constant vector
$c^*$ ($A_i$'s only depend on the values of $\alpha_{ij}$ and $\beta_{ij}$). Indeed, for this network,  by \eqref{eq:a} and \eqref{eq:b}, we have $A_1=1$, $A_2=-1$, and $A_3=-1$ for any choice of total-constant vector $c^*\in \mathbb{R}^2$. So,
by \eqref{eq:iinterval} and \eqref{eq:tinterval}, the left endpoint of the interval $I$ defined in \eqref{eq:tinterval} is always $0$. Hence,
the index $\tau$ defined in \eqref{eq:setk} is
 $\tau=1$.
 Note also, it is straightforward to check by the conservation laws (see Example \ref{ex:counter2}) that
 if $(\kappa^*, c^*)\in {\mathcal W}$, then
 $c_1^*\neq 0$ and $c_2^*\neq 0$. So, by
 \eqref{eq:a}--\eqref{eq:iec},
 there is only one index $\tau$ in $[\tau]$.
  Therefore, by Corollary \ref{cry:bd}, for any $(\kappa^*, c^*)\in {\mathcal W}$,
the left-hand side of the inequality \eqref{eq:sbd} is always
\textcolor{black}{
$$(\alpha_{\tau 2}-\alpha_{\tau 1})(\beta_{\tau 1}-\alpha_{\tau 1})=(\alpha_{1 2}-\alpha_{1 1})(\beta_{1 1}-\alpha_{1 1})=-1<0.$$}
That means ${\mathcal W}\cap {\mathcal B}_{2-react}=\emptyset$ for this example.
%Now we choose a total constant $c^* :=(c_1^*,c_2^*)=(-3,-\frac{11}{4})$ ($\in \mathbb{R}^2$). By \eqref{eq:a}, we have $A_1=1$ and $B_1=0$. By \eqref{eq:b},
%we have $A_2=-1$, $B_2=3$, $A_3=-1$ and $B_3=\frac{11}{4}$. So, it is seen that the set of indices defined in \eqref{eq:nindexj1} is ${\mathcal J}=\{1,2,3\}$, and there are three equivalence classes defined in \eqref{eq:iec}:
%$$[1]=\{1\},\;\; [2]=\{2\},\;\text{and}\;\; [3]=\{3\}.$$
%By \eqref{eq:varphi}, for each equivalence class, we have
% $$\varphi_1=\min\limits_{1\leq j\leq 2} \{\alpha_{1j}\}=1,\;\;\varphi_2=\min\limits_{1\leq j\leq 2} \{\alpha_{2j}\}=0,\;\text{and}\;\varphi_3=\min\limits_{1\leq j\leq 2} \{\alpha_{3j}\}=1,$$
% and so,  by \eqref{eq:gamma}, we have
 %$\gamma_{ij}=\alpha_{ij}-1$ for every $i\in \{1, 3\}$ and for every $j\in \{1, 2\}$, and $\gamma_{2j}=\alpha_{2j}$ for every $j\in \{1, 2\}.$
%\begin{align*}
%\begin{array}{ccc}
%\gamma_{11}=1&\gamma_{21}=2&\gamma_{31}=0\\
%\gamma_{12}=0&\gamma_{22}=0&\gamma_{32}=1
%\end{array}
%\end{align*}
%Hence, the set of indices defined in \eqref{eq:nindexi} is ${\mathcal H}=\{1,2,3\}$, and the open interval defined in \eqref{eq:tinterval} is $I=I_1\cap I_2\cap I_3=(0,\frac{11}{4})$, where
%$I_1=(0,+\infty)$ , $I_2=(0,3)$, and $I_3=(0,\frac{11}{4})$.
 %, and the set of indices defined in \eqref{eq:L} is ${\mathcal L}=\{2\}.$
%Therefore, by Theorem \ref{thm:nondegmultistable} and Corollary \ref{cry:bd}, only one of the three steady states is stable, which can be also verified by the computation presented in Example \ref{ex:counter2}.
 \end{example}
% \begin{example}
 %%\end{example}

\section{Univariate Polynomials}\label{sec:unipoly}
\begin{definition}\label{def:multiplicity}
 %Consider a polynomial $f \in {\mathbb R}[b, z]$, where $(b, z)\in  {\mathbb R}^{n}\times {\mathbb R}$. We say $(b^*, z^*)\in {\mathbb R}^{n}\times {\mathbb R}$ is a {\em singular point} of $f$ if
%$f(b^*, z^*)=0~, \;\frac{\partial f}{\partial z}(b^*, z^*)=0$, and
%$\frac{\partial f}{\partial b_i}(b^*, z^*)=0, \text{ for all } i=1, \dots, n$.
%We say $(b^*, z^*)$ is a {\em regular point} of $f$ if
%$f(b^*, z^*)=0$ and $(b^*, z^*)$ is not a singular point of $f$.
Suppose $p(z)$ is a univariate polynomial in ${\mathbb R}[z]$, we say  $z^*\in {\mathbb R}$ is a \defword{simple} solution of $p(z)=0$ if
$p(z^*)=0$ and $\frac{d p}{d z}(z^*)\neq 0$. We say  $z^*\in {\mathbb R}$ is a \defword{multiplicity-$N$ ($N\geq 2$)} solution  of $p(z)=0$ if
$p(z^*)=0~, \;  \frac{d p}{d z}(z^*)=0$, $\ldots, \frac{d^{N-1} p}{d z^{N-1}}(z^*)=0$ and $\frac{d^N p}{d z^N}(z^*)\neq 0.$
\end{definition}

\begin{lemma}\label{lm:uni}
%(e.g., see \cite[Lemma 3.6]{tx2020})
Let $p(z):=a_nz^n+\cdots+a_1z+a_0$ be a univariate polynomial in ${\mathbb R}[z]$. If
$z_1$ and $z_2$ are two distinct real solutions of
the equation
$p(z)=0$, and any real solution of $p(z)=0$
in the open interval $(z_1, z_2)$ has an even multiplicity, then
 we have $\frac{d p}{d z}(z_1)\frac{d p}{d z}(z_{2})\leq 0$.
\end{lemma}
\begin{proof}
We write
$p(z) =(z-z_1)(z-z_2)h(z)$, where $h(z)\in {\mathbb R}[z]$. Note that
    $$\frac{d p}{d z}(z) \;=\; (z-z_2)h(z) + (z-z_1)h(z) + (z-z_1)(z-z_{2})\frac{d h}{d z}(z).$$
    So,
    \begin{align}\label{eq:uni}
  \frac{d p}{d z}(z_1)\frac{d p}{d z}(z_{2}) \;=\; -(z_1-z_2)^{2}h(z_1)h(z_2).
  \end{align}
  Let ${\mathcal V}:=\{b\in (z_1, z_2)|p(b)=0\}$. Then, we can write
$h(z)=\tilde h(z)\prod_{b\in {\mathcal V}}(z-b)^{m_b}$, where $m_b$ is an even number, and
$\tilde h(z)=0$ has no real solutions in $(z_1, z_2)$ (which implies $\tilde  h(z_1)\tilde  h(z_2)\geq 0$). So,
by \eqref{eq:uni}, we have
 $$\frac{d p}{d z}(z_1)\frac{d p}{d z}(z_{2}) \;=\; -(z_1-z_2)^{2}\tilde  h(z_1)\tilde  h(z_2)\prod_{b\in {\mathcal V}}(z_1-b)^{m_b}(z_2-b)^{m_b}\leq 0.$$
 We are done.
 \end{proof}

%\begin{lemma}\label{lm:uni}
%(e.g., see \cite[Lemma 3.6]{tx2020})
%\textcolor{blue}{we might move it to somewhere else..Xiaoxian.}
%\textcolor{black}{
%Let $p(z):=a_nz^n+\cdots+a_1z+a_0$ be a univariate polynomial in ${\mathbb R}[z]$.} If the equation
%$p(z)=0$ has $N$ $(N\geq 2)$ distinct  real solutions $z_1<\cdots<z_N$, then
% we have $\frac{d p}{d z}(z_i)\frac{d p}{d z}(z_{i+1})\leq0$ for $i\in \{1, \ldots, N-1\}$.
%\end{lemma}

%\begin{proof}
%  The argument is the same with the proof of \cite[Lemma 3.6]{tx2020}.
%\end{proof}

%\begin{lemma}\label{lm:even-multi}
%\textcolor{red}{
%Let $p(z)$ be a univariate polynomial in ${\mathbb R}[z]$. If $z^{(1)}<z^{(2)}<z^{(3)}$ are solutions of $p(z)=0$, and $\frac{d p}{d z}(z^{(1)})\frac{d p}{d z}(z^{(3)})>0$, then $z^{(2)}$ is a multiplicity-$N$ solution, where $N$ is an odd number. }
%\end{lemma}

%\begin{proof}

%\end{proof}

\begin{lemma}\label{lm:signend}
 %\textcolor{blue}{this lemma should be stated in a more general way.--Xiaoxian}
 Let $p(z)$ be a univariate polynomial in ${\mathbb R}[z]$. Suppose the equation
$p(z)=0$ has exactly $N$ $(N>0)$ real solutions $z^{(1)}<\cdots<z^{(N)}$ in an open interval $(a, M)$ ($a\in {\mathbb R}$, and $M\in {\mathbb R}\cup \{+\infty\}$), where
$z^{(1)}$ is simple, and any multiple solution has an odd multiplicity.  If $p(a)\neq 0$,
 %\textcolor{blue}{need to think about how to deal with this condition}
 then
 ${\tt sign}(p(a)) = -{\tt sign}(\frac{d p}{d z}(z^{(1)}))$.
 \end{lemma}
 \begin{proof}
 We write
 \begin{align}\label{eq:tildegc}
 p(z)\;=\;C(z)\prod\limits_{i=1}^N(z-z^{(i)})^{m_i},
 \end{align}
 where $C(z)$ is a non-zero polynomial in ${\mathbb R}[z]$, $m_1=1$, and for any $i\geq 2$, $m_i$ is odd.
 Then, we have
 \begin{align*}
 \frac{d p}{d z}(z)\;=\;\frac{dC}{d z}(z)\prod\limits_{i=1}^N(z-z^{(i)})^{m_i}+
 C(z)\sum\limits_{i=1}^N m_i(z-z^{(i)})^{m_i-1}\prod\limits_{k\neq i}(z-z^{(k)})^{m_k}.
 \end{align*}
 Thus, we have
 \begin{align}\label{eq:dtildegc}
\frac{d p}{d z}(z^{(1)})\;=\;
 C(z^{(1)})\prod\limits_{k=2}^N(z^{(1)}-z^{(k)})^{m_k}.
 \end{align}
 Note that $C(a)\neq 0$ since $p(a)\neq 0$. Also, $C(z)$ does not change its sign over $(a,M)$ since $p(z)=0$ has exactly $N$ solutions in $(a,M)$. Hence,
for any $z\in (a, M)$, $C(z)$ has the same sign with
$C(z^{(1)})$. So, ${\tt sign}(C(a))={\tt sign}(C(z^{(1)}))$ since
$C(z)$ is continuous.
Therefore, by \eqref{eq:tildegc} and \eqref{eq:dtildegc},
the sign of $p(a)$ is equal to
$$
\begin{cases}
-{\tt sign}(C(a))= -{\tt sign}(C(z^{(1)}))=-{\tt sign}(\frac{d p}{d z}(z^{(1)})), & \text{if}\; N \;\text{is odd},\\
{\tt sign}(C(a))= {\tt sign}(C(z^{(1)}))=-{\tt sign}(\frac{dp}{d z}(z^{(1)})), & \text{if}\; N \;\text{is even}.
\end{cases}$$
We complete the proof.
%We are done.
\end{proof}

\section{Networks with one-dimensional stoichiometric subspaces}\label{sec:dim1}

In this section, we focus on the networks with one-dimensional stoichiometric subspaces.
We present a list of nice results on these networks.
 In Section \ref{sec:ss}, we show the steady-state system (augmented with the conservation laws) can be reduced to a univariate polynomial
 (see \eqref{eq:h1}--\eqref{eq:g}).
 In Section \ref{sec:consis}, we present a criterion for stability (Lemma \ref{lm:stable}).
 In Section \ref{sec:nss}, we show a nondegenerate steady state corresponds to a simple real root of a univariate polynomial (Lemma \ref{lm:nonde}).
 In Section \ref{sec:sign}, we present a sign condition (Theorem \ref{thm:sign}), which implies
 an upper bound of the number of stable positive steady states (Corollary \ref{cry:sign-stable}).
 In Section \ref{sec:apss}, we prepare a list of lemmas for the networks admitting at least one positive steady state.

  We remark that in the first author's previous work \cite{tx2020}, Lemma \ref{lm:stable}, Lemma \ref{lm:jh}, Lemma \ref{lm:jac}, and Lemma \ref{lm:nonde} are proved for networks with two reactions (possibly reversible). Those proofs can be directly to extend to any networks with one-dimensional stoichiometric subspaces. Also, Lemmas \ref{lm:jh}--\ref{lm:nonde} are implied by the results \cite[Propositions 9.1--9.2]{ConradiPantea_Multi} on the ``reduced Jacobian matrix"
\cite[Definition 9.8]{ConradiPantea_Multi}.
So, we omit the proofs here.  Note also Theorem \ref{thm:sign} is a straightforward generalization of \cite[Theorem 3.5]{tx2020}.

\subsection{Steady States}\label{sec:ss}
 If the stoichiometric subspace of a network $G$ \eqref{eq:network} is one-dimensional, then under Assumption \ref{assumption}, for every
$j\in \{1,\ldots,m\}$,
 there exists $\lambda_j \in {\mathbb R}$ such that
\begin{align}\label{eq:scalar}
\left(
\begin{array}{c}
\beta_{1j}-\alpha_{1j}\\
\vdots\\
\beta_{sj}-\alpha_{sj}
\end{array}
\right)\;=\;
\lambda_j\left(
\begin{array}{c}
\beta_{11}-\alpha_{11}\\
\vdots\\
\beta_{s1}-\alpha_{s1}
\end{array}
\right).
\end{align}
Note here, we have $\lambda_1=1$, and by the definition of chemical reaction network (which requires $(\alpha_{1j},\ldots,\alpha_{sj})\neq (\beta_{1j},\ldots,\beta_{sj})$ in \eqref{eq:network}), we have
$\lambda_j\neq 0$ for all $j\in \{2, \ldots, m\}$.
We substitute
\eqref{eq:scalar} into $f(x)$ in \eqref{eq:sys}, and
we have
\begin{align}\label{eq:f}
f_i \;=\;\left(\beta_{i1}-\alpha_{i1}\right) \sum\limits_{j=1}^m\lambda_j\kappa_j\prod\limits_{k=1}^s x_k^{\alpha_{kj}}, \;\; i=1, \ldots, s.
\end{align}
We define the steady-state system augmented with the conservation laws:
\begin{align}
h_1 ~&:=~ f_1\;=\;\left(\beta_{11}-\alpha_{11}\right)\sum\limits_{j=1}^m\lambda_j\kappa_j\prod\limits_{k=1}^s x_k^{\alpha_{kj}} \label{eq:h1}\\
h_i ~&:= ~(\beta_{i1}-\alpha_{i1})x_1 - (\beta_{11}-\alpha_{11})x_i - c_{i-1} \;(\text{here},\; c_{i-1}\in {\mathbb R}), \;\; 2\leq i\leq s.\label{eq:h}
\end{align}

We first solve $x_i$ from $h_i=0$ for $i\in \{2, \ldots, s\}$, and then we substitute the symbolic solution into $h_1$. We get
a polynomial
%For the system $h$ defined in \eqref{eq:h1}--\eqref{eq:h}, define
\begin{align}
%g(x_1) &~:=~ {\tt res}({\tt res}({\tt res}(h_1, h_2, x_2), h_3, x_3), \ldots, x_{s})
g(\kappa; c; x_1) &~:=~ h_1(x_1, \ldots, x_s)|_{x_2=\frac{\beta_{21}-\alpha_{21}}{\beta_{11}-\alpha_{11}}x_1-\frac{c_{1}}{\beta_{11}-\alpha_{11}}, \;\ldots,\; x_s=\frac{\beta_{s1}-\alpha_{s1}}{\beta_{11}-\alpha_{11}}x_1-\frac{c_{s-1}}
{\beta_{11}-\alpha_{11}}}\notag\\
&~=~ \left(\beta_{11}-\alpha_{11}\right)\sum\limits_{j=1}^m\lambda_j\kappa_jx_1^{\alpha_{1j}}\prod\limits_{k=2}^s (\frac{\beta_{k1}-\alpha_{k1}}{\beta_{11}-\alpha_{11}}x_1-\frac{c_{k-1}}{\beta_{11}-\alpha_{11}})^{\alpha_{kj}}. \label{eq:g}
\end{align}
Clearly, if for a rate-constant $\kappa^*$,
$x^*$ is a steady state in ${\mathcal P}_{c^*}$, then
$x^*$ is a common solution to the equations $h_1(x^*)=\ldots=h_s(x^*)=0$, and
the first coordinate of $x^*$, denoted by $x_1^*$, is a solution to the univariate equation $g(\kappa^*;c^*;x_1)=0$.
\subsection{Stability}\label{sec:consis}

%\begin{lemma}\label{lm:assumption}
%For any  network $G$, % in \eqref{eq:network} has exactly two reactions (i.e., $m=2$).
%$cap_{pos}(G)\leq 1$ if
%the stoichiometric matrix $\mathcal{N}$ is a zero matrix.
 %$\beta_{k1}-\alpha_{k1}=0$ for all $k=1, \ldots, s$.
%$G$ has at most one positive steady state in any stoichiometric compatibility class ${\mathcal P}_c$.
%\end{lemma}
%\begin{proof}
%By Lemma \ref{lm:scalar}, if the equality \eqref{eq:scalar} does not hold for any $\lambda \in {\mathbb R}\backslash\{0\}$, then
%$cap_{pos}(G)<1$. If  \eqref{eq:scalar} holds for some $\lambda \in {\mathbb R}\backslash\{0\}$, and
%if $\beta_{k1}-\alpha_{k1}=0$ for all $k=1, \ldots, s$, then %by Lemma \ref{lm:scalar},
%If the stoichiometric matrix ${\mathcal S}$ is a zero matrix, then  the
%conservation-law matrix $W$ is square and full rank. That means
%for any $c\in {\mathbb R}$, there is at most one positive steady state in the %stoichiometric compatibility class ${\mathcal P}_c$, see \eqref{eq:pc}.
%\end{proof}

%{\bf Assumption 1.}

  \begin{lemma}\label{lm:stable}\cite[Lemma 3.4]{tx2020}
If  the stoichiometric subspace of a network $G$ \eqref{eq:network} is one-dimensional, then a nondegenerate steady state $x^*$ is stable if and only if $\sum_{i=1}^s\frac{\partial f_i}{\partial x_i}|_{x=x^*}<0$, where
 $f_i$'s are the polynomials defined in \eqref{eq:f}.% if and only if .
\end{lemma}

\begin{example}\label{ex:counter2}
Consider the network in Example \ref{ex:counter}.
 It is straightforward to check that
the equality \eqref{eq:scalar} holds for $\lambda_1=1$ and $\lambda_2=-1$.
  Let $\kappa_1=1$, $\kappa_2=\frac{8}{3}$, $c_1=-3$, and $c_2=-\frac{11}{4}$.
  Then we have the steady-state system $h$:
\begin{align*}
h_1&\;=\;\left(\beta_{11}-\alpha_{11}\right)
\left(\lambda_1\kappa_1\prod\limits_{k=1}^s x_k^{\alpha_{k1}}+\lambda_2\kappa_2\prod\limits_{k=1}^s x_k^{\alpha_{k2}}\right)\;=\;
x_1x_3(x_1x^2_2-\frac{8}{3}x_3), \\
h_2& \;=\; (\beta_{21}-\alpha_{21})x_1 - (\beta_{11}-\alpha_{11})x_2 - c_{1}\;=\;-x_1 -x_2 +3, \;\;\;\text{and}\\
h_3 &\;=\; (\beta_{31}-\alpha_{31})x_1 - (\beta_{11}-\alpha_{11})x_3 - c_{2}\;=\;-x_1 -x_3 + \frac{11}{4}.
\end{align*}
By solving the equations $h_1(x)=h_2(x)=h_3(x)=0$, we find three nondegenerate positive steady states:
%\textcolor{red}{Zhishuo, please double check}
{\scriptsize
%[x1 = 2.577350269, x2 = 0.4226497308, x3 = 0.1726497308], [x1 = 2., x2 = 1., x3 = 0.7500000000], [x1 = 1.422649731, x2 = 1.577350269, x3 = 1.327350269]
\[x^{(1)}\approx (2.577,\;
    0.423,\;
    0.173), \;\;
x^{(2)}=(2,\;1,\;\frac{3}{4}),\;\;
x^{(3)}\approx(1.423,\;
   1.577,\;
     1.327). \]
    }
It is straightforward to check \textcolor{black}{by Lemma \ref{lm:stable}} that
only $x^{(2)}$ is stable.
Here, we complete  the computation  by {\tt Maple2020} \cite{maple}, see ``Example.mw" in Table \ref{tab:sup}.
\end{example}

\subsection{Nondegenerate Steady States}\label{sec:nss}

\begin{lemma}\label{lm:jh}\cite[Lemma 3.7]{tx2020}
Let $f$ and $h$ be the systems defined in \eqref{eq:f}--\eqref{eq:h}.
Denote  by $|{\rm Jac}_h|$ the determinant of the Jacobian matrix of $h$ with respect to $x$. Then, we have
\begin{align*}
|{\rm Jac}_h|%&\;=\; (\alpha_{11}-\beta_{11})^{s-2}\sum_{i=1}^s(\alpha_{i1}-\beta_{i1})\frac{\partial f_1}{\partial x_i}, \label{eq:jh}\\
%\;=\;(\beta_{11}-\alpha_{11})^{s-1}\sum_{i=1}^s\frac{\partial f_i}{\partial x_i}.
%&
\;=\;(\alpha_{11}-\beta_{11})^{s-1}\sum_{i=1}^s\frac{\partial f_i}{\partial x_i}.% \label{eq:cryjh}
\end{align*}
\end{lemma}

%\begin{corollary}\label{cry:jh}
%The determinant of Jacobian matrix of $h$ with respect to $x$ is equal to
%\begin{align}
%(\beta_{11}-\alpha_{11})^{s-1}\sum_{i=1}^s\frac{\partial f_i}{\partial x_i}.
%\end{align}
%\end{corollary}

%\begin{proof}
%The conclusion directly follows from Lemma \ref{lm:jh} and the equalities \eqref{eq:conf}.
%\end{proof}

%\begin{assumption}\label{assumption}
%The network $G$ admits at least one nondegenerate steady state (i.e., there exists $(\kappa^*; c^*)\in {\mathbb R}^{2}\times {\mathbb R}^{s-1}$ such that $h(x)=0$ has at least one solution $x^*\in {\mathbb R}^{s}$ and ${\rm Jac}_h|_{x=x^*}\neq 0$).
%\end{assumption}
%\begin{remark}\label{rmk:assumption}
%By Lemma \ref{lm:jh}, Assumption \ref{assumption} implies that $\beta_{11}-\alpha_{11}\neq 0$.
%\end{remark}

%\begin{remark}
%resultant is better?
%$g(x_1)$ in \eqref{eq:g} can be considered as the resultant.
%\end{remark}

\begin{lemma}\label{lm:jac}\cite[Lemma 3.8]{tx2020}
For the system $h$ defined in \eqref{eq:h1}--\eqref{eq:h} and the polynomial $g$ defined in \eqref{eq:g},
if for a rate-constant vector $\kappa^*$ and a total-constant vector $c^*$,  $x^*$ is a solution to $h_1(x^*)=\ldots =h_s(x^*)=0$, then
\begin{align*}%\label{eq:jac}
\left(\alpha_{11}-\beta_{11}\right)^{s-1}\pd{g}{x_1}(\kappa^*;c^*; x^*_1) \;=\;  |{\rm Jac}_h(x^*)|.
%$\frac{d g}{d x_1}(x^*_1) \;=\; \gamma {\rm Jac}_h(x^*)$, where $\gamma\in {\mathbb R}$.
\end{align*}
\end{lemma}

\begin{lemma}\label{lm:nonde}\cite[Lemma 3.11]{tx2020}
Given a network $G$ \eqref{eq:network}, suppose the stoichiometric subspace of $G$ is one-dimensional.
Let $h$ and $g$ be defined as in \eqref{eq:h1}--\eqref{eq:g}.
If for a rate-constant vector $\kappa^*$,  there is a steady state $x^*$ in the stoichiometric compatibility class ${\mathcal P}_{c^*}$, then the following statements are equivalent.
\begin{itemize}
\item[(i)] The steady state $x^*$ is nondegenerate.
\item[(ii)] $|{\rm Jac}_h(x^*)|\neq 0$.
\item[(iii)] $\pd{g}{x_1}(\kappa^*;c^*; x^*_1)\neq 0$.
\end{itemize}
\end{lemma}

\subsection{Sign condition}\label{sec:sign}

%\textcolor{blue}{keep these for a while...need to think about how to present--Xiaoxian}
\begin{theorem}\label{thm:sign}
%\cite[Theorem 3.5]{tx2020}
\textcolor{black}{
Given a network $G$,
%\textcolor{red}{Given $G\in {\mathcal G}$},
suppose the stoichiometric subspace of $G$ is one-dimensional.}
Let $h$ be the system defined in \eqref{eq:h1}--\eqref{eq:h}.
If for a rate-constant vector $\kappa^*$, $G$ has $N$ distinct positive steady states $x^{(1)}, \ldots, x^{(N)}$ in ${\mathcal P}_{c^*}$, where these steady states are ordered  according to their first coordinates (i.e., $x_1^{(1)}<\ldots<x_1^{(N)}$),
then  for $i\in \{1, \ldots, N-1\}$, we have $|{\rm Jac}_h(x^{(i)})||{\rm Jac}_h(x^{(i+1)})|\leq 0$.
In addition, if all these positive steady states are nondegenerate, then we have $|{\rm Jac}_h(x^{(i)})||{\rm Jac}_h(x^{(i+1)})|<0$.
\end{theorem}
\begin{proof}
  The argument  is the similar to the proof of \cite[Theorem 3.5]{tx2020}.
\end{proof}

\begin{corollary}\label{cry:sign-stable}
\textcolor{black}{
  Given a network $G$ \eqref{eq:network}, suppose the stoichiometric subspace of $G$ is one-dimensional. If
 for a rate-constant vector $\kappa^*$ and a total-constant vector $c^*$, $G$ has exactly  $N$  positive steady states $x^{(1)}, \ldots, x^{(N)}$ $(x_1^{(1)}<\ldots<x_1^{(N)})$ in the stoichiometric compatibility class ${\mathcal P}_{c^*}$,
 %and suppose the positive steady states are $x^{(1)}, \ldots, x^{(N)}$ such that $x_1^{(1)}<\ldots<x_1^{(N)}$,
  then
  \begin{itemize}
\item[(a)] for any $1\leq i\leq N-1$, $x^{(i)}$ and $x^{(i+1)}$ cannot be both stable, and
\item[(b)] at most $\lceil \frac{N}{2} \rceil$ of these positive steady states are stable.
 \end{itemize}
 }
\end{corollary}

\begin{proof}
(a)
 By Lemma \ref{lm:stable},  a nondegenerate steady state $x^*$ is stable if and only if $\sum_{i=1}^s\frac{\partial f_i}{\partial x_i}(x^*)<0$. By Lemma \ref{lm:jh}, we have $$|{\rm Jac}_h(x^*)|=(\alpha_{11}-\beta_{11})^{s-1}\sum_{i=1}^s\frac{\partial f_i}{\partial x_i}(x^*).$$
So, the conclusion follows from Theorem \ref{thm:sign}.
 %The conclusion follows from  Lemma \ref{lm:stable}, Lemma \ref{lm:jh},  and Theorem \ref{thm:sign}.

(b) The conclusion directly follows from (a).
\end{proof}
\subsection{Networks Admitting Positive Steady States}\label{sec:apss}
In this section, we prove a list of lemmas for the networks admitting at least one positive steady state. First, we will factor $g(\kappa;c; x_1)$ (see \eqref{eq:sg}) for a given total-constant vector $c^*$, and we deduce a polynomial
$q(\kappa; x_1)$ (see \eqref{eq:q}). In Lemmas \ref{lm:qroot}--\ref{lm:jactildeg}, we show that
the first coordinate of a (nondegenerate) positive steady state is a (simple) real solution of the equation $q(\kappa; x_1)=0$. Following from Lemma \ref{lm:qroot}, Lemma \ref{lm:assumption2} explains why Assumption \ref{assumption2} is reasonable (also see Remark \ref{rmk:assumption2} and Example \ref{ex:assumption2}). From the equation $q(\kappa; x_1)=0$, we can symbolically solve a carefully-chosen (see \eqref{eq:ell}) parameter $\kappa_{\ell}$, and we deduce a function
$\phi$ (see \eqref{eq:phi}). In Lemma \ref{lm:well-defined}, we show a list of properties of this function, which will be useful in the next two sections.

Suppose a network $G$ \eqref{eq:network} (with $s$ species and $m$ reactions) has one-dimensional stoichiometric subspace.
Recall Section \ref{sec:ss}.
By the steady-state system $h$ augmented with conservation laws defined in \eqref{eq:h1}-\eqref{eq:h}, we can deduce a polynomial $g(\kappa;c; x_1)$ in \eqref{eq:g}.
Suppose for a total constant-vector $c^*$, the notions $[i]$, $A_i$, $B_i$, $r$, $\varphi_k$, $\gamma_{kj}$, ${\mathcal J}$, ${\mathcal H}$, $I_i$, $I$, $\tau$, and ${\mathcal L}$ are defined in \eqref{eq:iec}--\eqref{eq:L}. For every $i\in \{1, \ldots, s\}$, we define
\begin{align}
Y_i(x_1) \;:=\;
\begin{cases}
\frac{1}{|A_i|}(A_ix_1 + B_i), & i\in {\mathcal J} \; (i.e., A_i\neq 0) \\
1, & i\not\in {\mathcal J} \; (i.e., A_i=0).
\end{cases}\label{eq:y}
\end{align}
For every $j\in \{1, \ldots, m\}$, we define
\begin{align}
{\mathcal C}_j \;:=\; \prod\limits_{i\in {\mathcal J}}|A_i|^{\alpha_{ij}}\prod\limits_{i\in \{1, \ldots, s\}\backslash{\mathcal J}}B_i^{\alpha_{ij}}.\label{eq:mathcalc}
%Y_k(x_1) \;:=\;
%\begin{cases}
%\frac{A_k}{|A_k|}x_1 + \frac{B_k}{|A_k|}, & k\in {\mathcal J} \; (i.e., A_k\neq 0) \\
%1, & k\not\in {\mathcal J} \; (i.e., A_k=0).
%\end{cases}\label{eq:y}
\end{align}
When $c=c^*$, %recall again $r$, $\varphi_k$, ${\mathcal J}$ and ${\mathcal H}$ defined in
%\eqref{eq:r}, \eqref{eq:varphi}, \eqref{eq:nindexj1} and \eqref{eq:nindexi}.
we can write $g(\kappa; c^*; x_1)$  in \eqref{eq:g} as
\begin{align}\label{eq:sg}
g(\kappa; c^*; x_1) \;=\; q(\kappa; x_1)\prod\limits_{k=1}^rY_k(x_1)^{\varphi_k},
\end{align}
where
\begin{align}
q(\kappa; x_1) &\;:=\;(\beta_{11}-\alpha_{11})\sum_{j=1}^m\lambda_j\kappa_j\prod\limits_{i\in {\mathcal J}}|A_i|^{\alpha_{ij}}\prod\limits_{i\in \{1, \ldots, s\}\backslash{\mathcal J}}B_i^{\alpha_{ij}}
\prod\limits_{k=1}^rY_k(x_1)^{\gamma_{kj}}, \notag \\
&\;=\;(\beta_{11}-\alpha_{11})\sum_{j=1}^m\lambda_j\kappa_j{\mathcal C}_j
\prod\limits_{k\in \{1, \ldots r\}\cap {\mathcal J}}Y_k(x_1)^{\gamma_{kj}} \notag \\
 &\;=\;(\beta_{11}-\alpha_{11})\sum_{j=1}^m\lambda_j\kappa_j{\mathcal C}_j
\prod\limits_{k\in {\mathcal H}}Y_k(x_1)^{\gamma_{kj}}. \label{eq:q}
\end{align}
%where again
In \eqref{eq:q}, the second equality follows from the definition of ${\mathcal C}_j$ \eqref{eq:mathcalc} and the fact that when $k\not \in {\mathcal J}$, $Y_k(x_1)=1$ (see \eqref{eq:y}). The last equality holds because when $k\not \in {\mathcal H}$,
$\gamma_{kj}=0$. In fact, if $k\not\in {\mathcal H}$, then by \eqref{eq:nindexi},
$\gamma_{k1}, \ldots, \gamma_{km}$ are all the same. So, by \eqref{eq:varphi} and \eqref{eq:gamma},
for any $j\in \{1, \ldots, m\}$, $\gamma_{kj}=0$.

Now, for the index $\tau$ defined in \eqref{eq:setk}, we define
 \begin{align}\label{eq:ell}
 %\gamma \; := \; \min \limits_{1\leq j\leq m}  \{\varphi(k,j)\}, \;\text{and}\;
 \ell\; :=\; \argmin \limits_{j\in \{1, \ldots, m\}}  \{\gamma_{\tau j}\}.
 \end{align}
 and for every $j\in \{1, \ldots, m\}$, define
 \begin{align}\label{eq:pj}
 %P_j(x_1)\;:=\; (x_1-a)^{-\varphi_k}\prod\limits_{i=1}^s (A_ix_1+B_i)^{\alpha_{ij}}.
P_j(x_1)\;:=\; {\mathcal C}_j
\prod\limits_{k\in {\mathcal H}}Y_k(x_1)^{\gamma_{kj}}
 \end{align}
 Then $q(\kappa; x_1)$ in \eqref{eq:q} can be written as
 \begin{align}\label{eq:pq}
 q(\kappa; x_1) &~=~ (\beta_{11}-\alpha_{11})\sum^m\limits_{j=1}\lambda_j\kappa_jP_{j}(x_1).
\end{align}
 %Note $P_j(x_1)$ is a polynomial in ${\mathbb R}[x_1]$, and $P_{\ell}(a)\neq 0$.
We solve $\kappa_{\ell}$ from $q(\kappa; x_1)=0$, and we  define  the symbolic solution as a function
\begin{align}\label{eq:phi}
%g(x_1) &~:=~ {\tt res}({\tt res}({\tt res}(h_1, h_2, x_2), h_3, x_3), \ldots, x_{s})
 \phi(\hat\kappa; x_1) &~:=~ -\frac{\sum^m\limits_{j=1, j\neq \ell}\lambda_j\kappa_jP_{j}(x_1)}{\lambda_{\ell}P_{\ell}(x_1)},
 %|_{x_i=\frac{\beta_{i1}-\alpha_{i1}}{\beta_{11}-\alpha_{11}}x_1-\frac{c_{i-1}}{\beta_{11}-\alpha_{11}}, \;2\leq i\leq s.}
%, \;\ldots,\; x_s=\frac{\beta_{s1}-\alpha_{s1}}{\beta_{11}-\alpha_{11}}x_1-\frac{c_{s-1}}
%{\beta_{11}-\alpha_{11}}}.
\end{align}
where $\hat \kappa := (\kappa_1, \ldots, \kappa_{\ell-1}, \kappa_{\ell+1}, \ldots, \kappa_{m})$.
Here, we recall that $\lambda_\ell\neq 0$ by \eqref{eq:scalar}.
For any $z\in {\mathbb R}$, we say $\phi(\hat\kappa; x_1)$ is \defword{well-defined at $z$} if
$P_{\ell}(z)\neq 0$.
\begin{lemma}\label{lm:ysign}
Given a network $G$,
   suppose the stoichiometric subspace of $G$ is one-dimensional.
   For a total-constant vector $c^*$,  let $I_i$ be the interval defined in \eqref{eq:iinterval}, and let $Y_i(x_1)$ be the function defined in \eqref{eq:y}.
   If $x^*_1\in I_i$,  then $Y_i(x^*_1)>0$.
\end{lemma}
\begin{proof}
Clearly, if $A_i=0$, then by \eqref{eq:y},  we have $Y_i(x^*_1)=1>0$.
Suppose $A_i\neq 0$.
 By \eqref{eq:y},
we have $Y_i(x_1)=\frac{1}{|A_i|}(A_ix_1+B_i)$. If  $x^*_1\in I_i$, then by \eqref{eq:iinterval},
$$A_ix^*_1+B_i\;>\;A_i(-\frac{B_i}{A_i})+B_i\;=\;0.$$
So, we have $Y_i(x^*_1)>0$.
\end{proof}

\begin{lemma}\label{lm:qroot}
Given a network $G$, suppose the stoichiometric subspace of $G$ is one-dimensional.
For a total-constant vector $c^*\in {\mathbb R}^{s-1}$,
let $q(\kappa;x_1)$ be the polynomial defined in \eqref{eq:q}.
 Then, for a  rate-constant vector $\kappa^*\in {\mathbb R}_{>0}^s$, $G$ has a positive steady state  in the stoichiometric compatibility class ${\mathcal P}_{c^*}$ if and only if
 $q(\kappa^*;x_1)=0$ has a solution in $\bigcap\limits_{i=1}^sI_i$, where $I_i$ is the open interval defined in \eqref{eq:iinterval}.
\end{lemma}
\begin{proof}
``$\Rightarrow$:"
Suppose $G$ has a positive steady state  $x^*:=(x^*_1, \ldots, x^*_s)$ in ${\mathcal P}_{c^*}$.
By \eqref{eq:h1}, \eqref{eq:h} and \eqref{eq:g}, we have $x^*_1$ is a solution of $g(\kappa^*;c^*;x_1)=0$.
By \eqref{eq:a}, \eqref{eq:b}, \eqref{eq:h1},  and \eqref{eq:h}, for any $i$ $(1\leq i\leq s)$,
$x^*_i=A_ix^*_1+B_i>0.$
So $x^*_1$ is contained in the interval $\bigcap\limits_{i=1}^sI_i$.
By  \eqref{eq:sg} and Lemma \ref{lm:ysign}, we know
$x^*_1$ is a solution of $q(\kappa^*;x_1)=0$.

``$\Leftarrow$:" Let $x^*_1$ be a solution of $q(\kappa^*;x_1)=0$ in $\bigcap\limits_{i=1}^sI_i$. Then, by \eqref{eq:sg},
 $x^*_1$ is a solution of $g(\kappa^*;c^*;x_1)=0$. For every $i$ $(2\leq i\leq s)$, let $x^*_i=A_ix^*_1+B_i$. Then, by \eqref{eq:iinterval}, we have $x^*_i>0$ for every $i$. By \eqref{eq:h1}, \eqref{eq:h} and \eqref{eq:g},
$x^*$ is a positive solution of $h_1=\cdots=h_s=0$ for $\kappa=\kappa^*$  and for $c=c^*$, i.e.,
it is a positive steady state in the stoichiometric compatibility class ${\mathcal P}_{c^*}$
\end{proof}

\begin{lemma}\label{lm:jactildeg}
Given a network $G$, suppose the stoichiometric subspace of $G$ is one-dimensional.
Let $g(\kappa; c; x_1)$ be the polynomial defined in \eqref{eq:g}. For a total-constant vector $c^*\in {\mathbb R}^{s-1}$,
let $q(\kappa;x_1)$ be the polynomial defined in \eqref{eq:q}.
If for a  rate-constant vector $\kappa^*\in {\mathbb R}_{>0}^s$, $G$ has a positive steady state $x^*$ in the stoichiometric compatibility class ${\mathcal P}_{c^*}$, then $\pd{g}{x_1}(\kappa^*;c^*;x^*_1)$ has the same sign with
$\pd{q}{x_1}(\kappa^*;x^*_1)$,
%\begin{align}\label{eq:jactildeg}
%\sum_{i=1}^s (\beta_{i1}-\alpha_{i1})(\alpha_{i1}-\alpha_{i2})\prod\limits_{k\neq i}{x_k^*},
%\end{align}
and additionally, the steady state
 $x^*$ is stable  if and only if $\pd{q}{x_1}(\kappa^*;x^*_1)<0$.
%If $x^*\in {\mathbb R}^s_{>0}$ is a solution to $h_1(x)=\ldots =h_s(x)=0$, then
%$\frac{d g}{d x_1}(x^*_1)$ has the same sign with $\frac{d g}{d x_1}(x^*_1)$.
\end{lemma}

\begin{proof}
%Let $Y(x_1):=\prod\limits_{k=1}^rY_k(x_1)^{\varphi_k}$, \textcolor{black}{then $g(\kappa^*; c^*; x_1)=Y(x_1)q(\kappa^*; x_1)$. Take the derivative of $g(\kappa^*; c^*; x_1)$, we have $\pd{g}{x_1}(\kappa^*;c^*;x_1)=Y'(x_1)q(\kappa; x_1)+Y(x_1)\pd{q}{x_1}(\kappa^*; x_1).$ Then evaluate at $x_1=x_1^*$ :
By \eqref{eq:sg}, we have
$\pd{g}{x_1}(\kappa^*;c^*;x_1^*)=\pd{q}{x_1}(\kappa^*; x_1^*)\prod\limits_{k=1}^rY_k(x_1^*)^{\varphi_k}$.
By the proof of Lemma \ref{lm:qroot}, $x_1^*\in \bigcap\limits_{i=1}^sI_i$, where $I_i$ is defined in \eqref{eq:iinterval}.
By Lemma \ref{lm:ysign}, we have $Y_k(x_1^*)>0$ for any $k$.
%where $Y(x_1^*)=\prod\limits_{k=1}^rY_k(x_1^*)^{\varphi_k}>0$.
Therefore, ${\tt sign}(\pd{g}{x_1}(\kappa^*;c^*;x_1^*))={\tt sign}(\pd{q}{x_1}(\kappa^*; x_1^*))$. By Lemma \ref{lm:stable}, Lemma \ref{lm:jh} and Lemma \ref{lm:jac}, the steady state $x^*$ is stable if and only if $\pd{g}{x_1}(\kappa^*;c^*;x_1^*)<0$, and thus, we have the conclusion.
\end{proof}

 \begin{lemma}\label{lm:assumption2}
Given a network $G$,
   suppose the stoichiometric subspace of $G$ is one-dimensional.
   If for a total-constant vector $c^*\in {\mathbb R}^{s-1}$,
   the set of indices ${\mathcal H}$ defined \eqref{eq:nindexi} is empty,
   %either
   %$\varphi(j,1)=\varphi(j,2)=\cdots=\varphi(j,m)$
   %or
   %$(\alpha_{i1}, \ldots, \alpha_{im})=(\beta_{i1}, \ldots, \beta_{im})$
   then for any rate-constant vector $\kappa\in {\mathbb R}^{s}$,
   $G$ has
   %no nondegenerate positive steady state in the stoichiometric compatibility
    %class ${\mathcal P}_{c^*}$.
     either no positive steady states, or infinitely many positive steady states in the stoichiometric compatibility
    class ${\mathcal P}_{c^*}$.
\end{lemma}

\begin{proof}
Suppose for a rate-constant vector $\kappa^*\in {\mathbb R}^{s}$,
$x^*$ is a positive steady state in the stoichiometric compatibility class ${\mathcal P}_{c^*}$.
By the proof of Lemma \ref{lm:qroot},
 the first coordinate of $x^*$, say $x_1^*$, is a solution of the equation $q(\kappa^*; x_1)=0$  in $\bigcap\limits_{i=1}^sI_i$, where $I_i$ is the open interval defined in \eqref{eq:iinterval}.
If ${\mathcal H}=\emptyset$, then by \eqref{eq:q},
we have
$$q(\kappa^*; x_1^*)\;=\; (\beta_{11}-\alpha_{11})\sum_{j=1}^m\lambda_j\kappa^*_j{\mathcal C}_j.$$
By Assumption \eqref{assumption}, $\beta_{11}-\alpha_{11}\neq 0$.
So, we must have  $\sum_{j=1}^m\lambda_j\kappa^*_j{\mathcal C}_j=0$.
Note also $\bigcap\limits_{i=1}^sI_i$ is nonempty since it contains $x^*_1$.
Therefore, any point in $\bigcap\limits_{i=1}^sI_i$ is a solution of $q(\kappa^*; x_1)=0$, and so, the conclusion follows from
Lemma \ref{lm:qroot}.
%$(y_1, A_2y_1+B_2, \ldots, A_sy_1+B_s)$ is a positive steady state.
%Therefore,
%$$\frac{\partial g}{\partial x_1}(\kappa^*; c^*; x_1^*)\;=\; (\beta_{11}-\alpha_{11})\sum_{k=1}^r\varphi_kY_k(x^*_1)^{-1}\prod\limits_{k=1}^rY_k(x^*_1)^{\varphi_k}\sum_{j=1}^m\lambda_j\kappa^*_j{\mathcal C}_j=0.$$
%By Lemma \ref{lm:jac} and Lemma \ref{lm:nonde}, we conclude that $x^*$ is not nondegenerate.
\end{proof}

%\begin{corollary}\label{cry:assumption2}
%Given a network $G$,
 %  suppose the stoichiometric subspace of $G$ is one-dimensional.
  % If for a total-constant vector $c^*\in {\mathbb R}^{s-1}$,  we have %either
   %$\varphi(j,1)=\varphi(j,2)=\cdots=\varphi(j,m)$
   %or
  % ${\mathcal H}=\emptyset$, where ${\mathcal H}$ is defined in \eqref{eq:nindexi},
  % then for any rate-constant vector $\kappa^*\in {\mathbb R}^{s}$,
  % $G$ has
   %no nondegenerate positive steady state in the stoichiometric compatibility
    %class ${\mathcal P}_{c^*}$.
   %  either no positive steady states, or has infinitely many positive steady states in the stoichiometric compatibility
    %class ${\mathcal P}_{c^*}$.
%\end{corollary}
%\textcolor{red}{
%In the discussion below, we will be only interested in the networks $G$ such that $0<cap_{pos}(G)<+\infty$, which
%implies there exists one index $k\in \{1, \ldots, s\}$ such that  $\alpha_{kj}$'s $(j\in \{1, \ldots, m\})$ are not all the same and
%$(\alpha_{k1}, \ldots, \alpha_{km})\neq (\beta_{k1}, \ldots, \beta_{km})$. Without loss of generality, we assume this
%index is $k=1$.
%So, besides Assumption \ref{assumption}, we
%make the following assumption.

\begin{remark}\label{rmk:assumption2}
In this remark, we explain why Assumption \ref{assumption2} is reasonable.
If for a  rate-constant vector $\kappa^*$, $G$ has $N$ positive steady states in  ${\mathcal P}_{c^*}$, where $0<N<+\infty$, then
Lemma \ref{lm:assumption2} guarantees that there exists one index in ${\mathcal H}$, say $k$.
If $1\not\in {\mathcal H}$, then we can relabel the species $X_1, \ldots, X_s$ as
$\tilde X_1, \ldots, \tilde X_s$ such that $\tilde X_1 = X_k$, $\tilde X_k = X_1$ and for any $i\not\in \{1, k\}$,
$X_i=\tilde X_i$. We illustrate this relabelling  by Example \ref{ex:assumption2}.
%Therefore,  Assumption \ref{assumption2}.
\end{remark}

 %\textcolor{blue}{need to put one example}

\begin{example}\label{ex:assumption2}
 Consider the following network $G$
 $$ X_1+2X_2+X_3+2X_4\xrightarrow{\kappa_1} 2X_1+X_2+3X_4,\;\;\;
 X_1+2X_3+X_4\xrightarrow{\kappa_2} X_2+3X_3. $$
 %It is easy to check the following results in this network.
 The conservation-law equations are
 \begin{align}\label{eq:cl1}
 -x_1-x_2-c_1=0, \;\; -x_1-x_3-c_2=0,\;\; x_1-x_4-c_3=0.
 \end{align}
\textcolor{black}{
We choose a total-constant vector $c^*:=(c_1^*, c_2^*, c_3^*)=(-3, 1, 2)$ ($\in \mathbb{R}^3$). It is straightforward to check that the set of indices defined in \eqref{eq:nindexj1} is ${\mathcal J}=\{1,2, 3, 4\}$, and there are $4$ equivalence classes defined in \eqref{eq:iec}:
$$[1]=\{1\},\; [2]=\{2\}, \; [3]=\{3\}, \; [4]=\{4\}.$$
By \eqref{eq:varphi}, for each equivalence class, we have
\begin{align*}
 \begin{array}{cc}
 \varphi_1=\min\limits_{1\leq j\leq 2} \{\alpha_{1j}\}=1,\;&\varphi_2=\min\limits_{1\leq j\leq 2} \{\alpha_{2j}\}=0, \\ \varphi_3=\min\limits_{1\leq j\leq 2} \{\alpha_{3j}\}=1, \;&\varphi_4=\min\limits_{1\leq j\leq 2} \{\alpha_{4j}\}=1,
 \end{array}
 \end{align*}
 %$$\varphi_1=\min\limits_{1\leq j\leq 2} \{\alpha_{1j}\}=1,\;\varphi_2=\min\limits_{1\leq j\leq 2} \{\alpha_{2j}\}=0, \varphi_3=\min\limits_{1\leq j\leq 2} \{\alpha_{3j}\}=1, \text{and}\;\varphi_4=\min\limits_{1\leq j\leq 2} \{\alpha_{4j}\}=1,$$
 and so,  by \eqref{eq:gamma}, we have
 \begin{align*}
 \begin{array}{cccc}
 \gamma_{11}=0, &\gamma_{21}=2 ,&\gamma_{31}=0 ,&\gamma_{41}=1, \\
 \gamma_{12}=0, &\gamma_{22}=0 ,&\gamma_{32}=1 ,&\gamma_{42}=0. \\
 \end{array}
 \end{align*}
 }Hence, the set of indices defined in \eqref{eq:nindexi} is ${\mathcal H}=\{2, 3, 4\}$, and $1\not\in {\mathcal H}$.
However, we can relabel $X_1$ and $X_4$ as what is suggested in Remark \ref{rmk:assumption2}. Consider
the network after relabelling:
 $$ \tilde X_4+2\tilde X_2+ \tilde X_3+2 \tilde X_1\xrightarrow{\kappa_1} 2\tilde X_4+\tilde X_2+3 \tilde X_1,\;\;\;
 \tilde X_4+2 \tilde X_3+\tilde X_1\xrightarrow{\kappa_2}\tilde X_2+3 \tilde X_3. $$
 For this network,
 the conservation-law equations are
 \begin{align}\label{eq:cl2}
 -\tilde x_1-\tilde x_2-\tilde c_1=0, \;\; -\tilde x_1-\tilde x_3-\tilde c_2=0,\;\; \tilde x_1-\tilde x_4-\tilde c_3=0.
 \end{align}
 Comparing \eqref{eq:cl1} and \eqref{eq:cl2}, and noting that
 $x_1=\tilde x_4$, $x_4=\tilde x_1$, $x_2=\tilde x_2$, and $x_3=\tilde x_3$,
 we have $\tilde c_1=c_1+c_3$, $\tilde c_2=c_2+c_3$, and  $\tilde c_3=-c_3$.
 So, for $c^*=(-3, 1, 2)$, the total-constant vector  becomes
 $\tilde c^*=(-1, 3, -2)$ after the relabelling.
 It is straightforward to check that the set of indices defined in \eqref{eq:nindexi} is $\tilde {\mathcal H}=\{1, 2,3\}$ \textcolor{black}{for $\tilde c^*$}. Now, we do have $1\in \tilde {\mathcal H}$.
\end{example}

 %Let $k$ be the index in \eqref{eq:setk}.

\begin{lemma}\label{rmk:mathcalc}
Given a network $G$,
   suppose the stoichiometric subspace of $G$ is one-dimensional.
   For a total-constant vector $c^*$,  let ${\mathcal C}_j$ be the number defined in \eqref{eq:mathcalc}.
If for a  rate-constant vector $\kappa^*$, $G$ has at least one positive steady state $x^*=(x^*_1, \ldots, x^*_s)$ in the stoichiometric compatibility class ${\mathcal P}_{c^*}$, then for any $j\in \{1, \ldots, m\}$, the number ${\mathcal C}_j$  is positive.
\end{lemma}
\begin{proof}
By \eqref{eq:a}, \eqref{eq:b}, \eqref{eq:h1}, and \eqref{eq:h}, for any $i$ $(1\leq i\leq s)$,
$x^*_i=A_ix^*_1+B_i.$
 If $i\in \{1, \ldots, s\}\backslash{\mathcal J}$, then we have $A_i=0$ by \eqref{eq:nindexj1}, and so, we have
$B_i=A_ix^*_1+B_i=x^*_i>0$.
\end{proof}
\begin{lemma}\label{lm:well-defined}
Given a network $G$,
   suppose the stoichiometric subspace of $G$ is one-dimensional.
   For a total-constant vector $c^*$, let $I$ be the open interval defined in \eqref{eq:tinterval}, and let $\phi(\hat \kappa; x_1)$ be the function defined in \eqref{eq:phi}.
   If for a  rate-constant vector $\kappa^*$, $G$ has a positive steady state $x^*$ in the stoichiometric compatibility class ${\mathcal P}_{c^*}$, then we have the following results.
   \begin{itemize}
%\item[(1)] For any $j\in \{1, \ldots, m\}$, ${\mathcal C}_j>0$ .
\item[(1)] For any $z\in I$,
$\phi(\hat \kappa; x_1)$ is well-defined at $z$.
\item[(2)] For any $k\in {\mathcal H}\backslash \{\tau\}$,
we have $Y_k(-\frac{B_{\tau}}{A_{\tau}})>0$, where ${\mathcal H}$, $\tau$ and $Y_k$ are defined in \eqref{eq:nindexi}, \eqref{eq:setk} and \eqref{eq:y}, and recall by \eqref{eq:setk} that $-\frac{B_{\tau}}{A_{\tau}}$ is the left endpoint of $I$.
\item[(3)] The function $\phi(\hat \kappa; x_1)$ is well-defined at the
left endpoint of $I$.
\end{itemize}
\end{lemma}
\begin{proof}
%(1)\; For any $i\in {1, \ldots, s}$, if $A_i=0$, then we have
%$B_i=x^*_i=A_ix^*_1+B_i>0$. So, by \eqref{eq:mathcalc}, we have
%${\mathcal C}_j>0$ for any $j\in \{1, \ldots, m\}$.

(1)\; By \eqref{eq:pj}, we have $P_{\ell}(x_1)={\mathcal C}_{\ell}
\prod\limits_{k\in {\mathcal H}}Y_k(x_1)^{\gamma_{k\ell}}$.
Recall  that $I=\bigcap\limits_{k\in {\mathcal H}}I_k$ \eqref{eq:tinterval}, where
$I_k$ is defined in \eqref{eq:iinterval}.
So, by   Lemma \ref{lm:ysign} and Lemma \ref{rmk:mathcalc},
for any $z\in I$, we have $P_{\ell}(z)>0$, and thus, $\phi(\hat \kappa; x_1)$ is well-defined at $z$.

(2)\;
Since $G$ has at least one positive steady state $x^*$ in the stoichiometric compatibility class ${\mathcal P}_{c^*}$,
by Lemma \ref{lm:qroot}, the interval $\bigcap\limits_{i=1}^sI_i$ is nonempty.
Note that by \eqref{eq:tinterval}, we have  $\bigcap\limits_{i=1}^sI_i\subset I$.
 Thus, $I$ is also nonempty.
%By \eqref{eq:setk}, the left endpoint $a$ of
 %$I$ is equal to $-\frac{B_{\tau}}{A_{\tau}}$.
 For any $k\in {\mathcal H}\backslash \{\tau\}$,  recall
 $Y_k(x_1)=\frac{1}{|A_k|}(A_kx_1+B_k)$ \eqref{eq:y}.
 We show
below that $Y_k(-\frac{B_{\tau}}{A_{\tau}})=\frac{1}{|A_k|}(B_k-A_k\frac{B_\tau}{A_\tau})>0$. In fact, we have the following cases.
\begin{itemize}
\item If $A_k>0$, then by \eqref{eq:iinterval}, $I_k=(-\frac{B_k}{A_k}, +\infty)$. Note that
$a=-\frac{B_\tau}{A_\tau}$ is the left endpoint of $I$ and $I\subset I_k$ by \eqref{eq:tinterval}. So,
 we have  $-\frac{B_\tau}{A_\tau}>-\frac{B_k}{A_k}$, and hence,
$$B_k-A_k\frac{B_\tau}{A_\tau}\;>\;B_k-A_k\frac{B_k}{A_k}\;=\;0.$$
%\item If $A_j=0$, then $B_j=x_j^{(1)}=x_j^{(2)}=x_j^{(3)}>0$.
\item If $A_k<0$, then by \eqref{eq:iinterval}, $I_k=(0, -\frac{B_k}{A_k})$. So, we have  $-\frac{B_\tau}{A_\tau}<-\frac{B_k}{A_k}$, and hence,
$$B_k-A_k\frac{B_\tau}{A_\tau}\;>\;B_k-A_k\frac{B_k}{A_k}\;=\;0.$$
\end{itemize}

(3)\;
 By \eqref{eq:varphi}, \eqref{eq:gamma} and \eqref{eq:ell}, we have
$\gamma_{\tau \ell}=0$. So, by \eqref{eq:pj},
we have $$P_{\ell}(x_1)\;=\;{\mathcal C}_{\ell}
\prod\limits_{k\in {\mathcal H}\backslash \{\tau\}}Y_k(x_1)^{\gamma_{k\ell}}.$$
So, by Lemma \ref{rmk:mathcalc} and (2), we have $P_{\ell}(a)>0$. Thus, $\phi(\hat \kappa; x_1)$ is well-defined at $a$.
 \end{proof}

\section{Nondegeneracy Conjecture}\label{sec:nc}
In this section, the goal is to prove Conjecture \ref{conj:nc} (Nondegeneracy Conjecture) for the networks with one-dimensional stoichiometric subspaces, see Theorem \ref{thm:nc}.
\begin{theorem}\label{thm:nc}
%\textcolor{red}{nondegenerate conjecutre}
%Suppose $G\in {\mathcal G}_0$, or, suppose $G\in {\mathcal G}_1\cup {\mathcal G}_2$ and $G$ has up to $2$ species.
Given a network $G$,
if the stoichiometric subspace of $G$ is one-dimensional, and if $cap_{pos}(G)<+\infty$, then $cap_{nondeg}(G)=cap_{pos}(G)$.
\end{theorem}

Below, we first present a list of lemmas (with examples). The lemmas for proving Theorem \ref{thm:nc} are also very useful in Section \ref{sec:mstab} (e.g., see the proof of Lemma \ref{lm:nobd}). After the proof, we provide Corollary \ref{cry:ww'}, which will play a key role in the next section when we prove the part (b) of Theorem \ref{thm:bdm} (the main result).
%Another important result following from Theorem \ref{thm:nc} is Theorem \ref{thm:nondegmultistable}, which implies the part (a) of Theorem \ref{thm:bdm}. We remark that Theorem \ref{thm:nondegmultistable} generalizes \cite[Theorem 3.16]{tx2020}.
\begin{lemma}\label{lm:phi}
%Let $g(\kappa; c; x_1)$ be the polynomial defined in \eqref{eq:g}.
Given a network $G$,
   suppose the stoichiometric subspace of $G$ is one-dimensional.
For a
total-constant vector $c^*$, %${\mathcal H}=(a, M)$ is the interval defined in \eqref{eq:tinterval},
let $q(\kappa; x_1)$ be the polynomial defined in \eqref{eq:q}, let  $\ell$ be the index defined in \eqref{eq:ell}
and
let $\phi(\hat \kappa; x_1)$ be the rational function defined in \eqref{eq:phi}.
If for a rate-constant vector $\kappa^*$,
 $x_1^*$  is a  multiplicity-$N$ ($N\geq 2$) solution of
 $q(\kappa^*;  x_1)=0$, and if $\phi(\hat \kappa; x_1)$ is well-defined at $x_1^*$, then
 we have
 $\phi(\hat \kappa^*; x_1^*)=\kappa_{\ell}^*$, $\frac{\partial \phi}{\partial x_1}(\hat \kappa^*; x_1^*)=0$, $\ldots, \frac{\partial^{N-1} \phi}{\partial x_1^{N-1}}(\hat \kappa^*; x_1^*)=0$  and $\frac{\partial^N \phi}{\partial x_1^N}(\hat \kappa^*; x_1^*)\neq 0$, where $\hat \kappa^* = (\kappa^*_1, \ldots, \kappa^*_{\ell-1}, \kappa^*_{\ell+1}, \ldots, \kappa^*_{m})$.
 % where
 %$\phi(x_1) := \tilde \phi(\kappa^*_2,\ldots,\kappa^*_m; c^*; x_1)$.
 \end{lemma}

 \begin{proof}

Define $\mu(\kappa_{\ell}, x_1) := q(\kappa^*_1, \ldots, \kappa^*_{\ell-1},\kappa_{\ell},\kappa^*_{\ell+1},\ldots,\kappa^*_m, x_1)$.
If for  $\kappa=\kappa^*$, $x_1^*$ is a  multiplicity-$N$ solution of
 $q(\kappa^*; x_1)=0$, then
	\begin{enumerate}[(I)]
	\item $\mu(\kappa_{\ell}^*, x_1^*)=0$,
	\item  $\pd{\mu}{x_1}(\kappa_{\ell}^*, x_1^*)=0$, \ldots, $\pd{^{N-1}\mu}{x_1^{N-1}}(\kappa_{\ell}^*, x_1^*) = 0$, and
	\item  $\pd{^N\mu}{x_1^N}(\kappa_{\ell}^*, x_1^*) \neq 0$.
	\end{enumerate}
%Let $I$ be the interval defined in \eqref{eq:tinterval}.
 For any $x_1\in {\mathbb R}$ such that $\phi(\hat \kappa; x_1)$ is well-defined, by \eqref{eq:pq} and \eqref{eq:phi}, we have
 $\phi(\hat \kappa^*; x_1^*)=\kappa_{\ell}^*$ and
\begin{align} \label{eq:mu=0}
\mu(\phi(\hat \kappa^*; x_1),  x_1)~=~0. %\quad {\rm for~} x_1 \in (a-\sigma, M)~.
\end{align}
Take the derivative of equation~\eqref{eq:mu=0}, via the chain rule:
\begin{align} \label{eq:deriv-1}
\pd{\mu}{\kappa_{\ell}} \left( \phi(\hat \kappa^*; x_1),x_1 \right) \pd{\phi}{x_1}(\hat\kappa^*;x_1) + \pd{\mu}{x_1}(\phi(\hat \kappa^*;x_1),x_1) ~=~0~.
\end{align}
\textcolor{black}{
Evaluating this equation at $x_1=x_1^*$, and recalling that $\phi(\hat \kappa^*;x_1^*)=\kappa_l^*$, we obtain:
\begin{align}\label{eq:deriv-2}
\pd{\mu}{\kappa_l}(\kappa_l^*,x_1^*)\pd{\phi}{x_1}(\hat \kappa^*;x_1^*)+ \pd{\mu}{x_1}(\kappa_l^*,x_1^*) ~=~0~.
\end{align}
Note that by \eqref{eq:pq},
$$\pd{\mu}{\kappa_l}(\kappa_l^*,x_1^*)=\pd{q}{\kappa_l}(\kappa^*;x_1^*)=(\beta_{11}-\alpha_{11})\lambda_{\ell}P_{\ell}(x^*_1)\neq 0.$$
(By Assumption \ref{assumption}, $\beta_{11}-\alpha_{11}\neq 0$. Recall $\lambda_{\ell}\neq 0$ by \eqref{eq:scalar}. And
$P_{\ell}(x^*_1)\neq 0$ because $\phi(\hat \kappa; x_1)$ is well-defined at $x^*_1$.)
 Note also $\pd{\mu}{x_1}(\kappa_l^*,x_1^*)=0$, by (II).  Thus, the equality \eqref{eq:deriv-2} implies $\pd{\phi}{x_1}(\hat \kappa^*;x_1^*)=0$.
Next,  take another derivative, applying the chain rule to equation~\eqref{eq:deriv-1}, and then evaluate at $x_1=x_1^*$:
{\tiny
\begin{align*}
\pd{\mu}{\kappa_l}(\kappa_l^*,x_1^*) \pd{^2\phi}{x_1^2}(\hat \kappa^*;x_1^*) +\pd{^2 \mu}{\kappa_l^2} (\kappa_l^*,x_1^*) \left(  \pd{\phi}{x_1}(\hat \kappa^*;x_1^*)  \right)^2
+ 2\pd{^2 \mu}{x_1 \partial \kappa_l} (\kappa_l^*,x_1^*) \pd{\phi}{x_1}(\hat \kappa^*;x_1^*) + \pd{^2 \mu}{x_1^2} (\kappa_l^*,x_1^*)=0.
\end{align*}
}Thus, we deduce from $\pd{\phi}{x_1}(\hat \kappa^*;x_1^*)=0$ that  $\pd{\mu}{\kappa_l}(\kappa_l^*,x_1^*) \pd{^2\phi}{x_1^2}(\hat \kappa^*;x_1^*)+ \pd{^2 \mu}{x_1^2} (\kappa_l^*,x_1^*)= 0$.
Similarly, for any $3\leq i\leq N$, we can deduce
 $\pd{\mu}{\kappa_l}(\kappa_l^*,x_1^*) \pd{^i\phi}{x_1^i}(\hat \kappa^*;x_1^*)+ \pd{^i \mu}{x_1^i} (\kappa_l^*,x_1^*)= 0$.
 By (II)  and (III), we have the conclusion.
 }
\end{proof}

\begin{lemma} \label{lm:even}
Given a network $G$,
   suppose the stoichiometric subspace of $G$ is one-dimensional.
\textcolor{black}{
For a total-constant vector $c^*$, let $q(\kappa;x_1)$ be the polynomial defined in \eqref{eq:q}, let  $\ell$ be the index defined in \eqref{eq:ell} and let $\phi(\hat \kappa; x_1)$ be the rational function defined in \eqref{eq:phi}.
If for a rate-constant vector $\kappa^*$,
 $x_1^*$  is a multiplicity-$2N$ ($N\in {\mathbb Z}_{>0}$) solution of
 $q(\kappa^*;  x_1)=0$, and if $\phi(\hat \kappa; x_1)$ is well-defined at $x_1^*$,
%If for a rate-constant vector $\kappa^*$ and a total-constant vector $c^*$,
% $x_1^*$ is a  multiplicity-$2N$ ($N\in {\mathbb Z}_{>0}$) solution of
 %$g(x_1)=0$,
 then for every $\epsilon>0$,  there exists $\delta \in {\mathbb R}$ such that for all $\delta' \in (0,|\delta|)$,
 \begin{itemize}
 \item[(1)] for
 $\kappa= (\kappa_1^*,\ldots, \kappa^*_{\ell-1}, \kappa_{\ell}^*+{\rm sign}(\delta)\delta', \kappa_{\ell+1}^*, \ldots, \kappa_m^*)$, the equation $q(\kappa;x_1)=0$ has two  simple real
solutions $x_1^{(1)}$ and $x_1^{(2)}$, for which $0<x_1^{(1)}-x_1^*<\epsilon$ and $-\epsilon<x_1^{(2)}- x_1^*<0$, and
\item[(2)]
for $\kappa= (\kappa_1^*,\ldots, \kappa^*_{\ell-1}, \kappa_{\ell}^*-{\rm sign}(\delta)\delta', \kappa_{\ell+1}^*, \ldots, \kappa_m^*)$, the equation $q(\kappa;x_1)=0$ has no real solution in the interval $(x_1^*-\epsilon, x_1^*+\epsilon)$.
\end{itemize}
}
\end{lemma}

\begin{proof}
\textcolor{black}{
By Lemma \ref{lm:phi}, we have
$$\phi(\hat \kappa^*;x_1^*)=\kappa_l^*, \pd{\phi}{x_1}(\hat \kappa^*;x_1^*)=0, \ldots, \pd{^{2N-1}\phi}{x_1^{2N-1}}(\hat \kappa^*;x_1^*)=0, \; \text{and}\; \pd{^{2N}\phi}{x_1^{2N}}(\hat \kappa^*;x_1^*)\neq 0.$$
Here, we recall that $\hat \kappa^* = (\kappa^*_1, \ldots, \kappa^*_{\ell-1}, \kappa^*_{\ell+1}, \ldots, \kappa^*_{m})$.
For any $\epsilon>0$,
without loss of generality, we assume $\pd{^{2N}\phi}{x_1^{2N}}(\hat \kappa^*;x_1)\neq 0$ for any
$x_1\in (x_1^*-\epsilon, x_1^*+\epsilon)$.
So, the Taylor expansion of $\phi(\hat \kappa^*;x_1)$ over $(x_1^*-\epsilon, x_1^*+\epsilon)$ is
{\footnotesize
\begin{align}
\phi(\hat \kappa^*;x_1)\;&=\;\phi(\hat \kappa^*;x_1^*)+\sum_{k=1}^{2N-1}\frac{1}{k!}\pd{^{k}\phi}{x_1^{k}}(\hat \kappa^*;x_1^*)(x_1-x^*_1)^k
+\frac{1}{(2N)!}\pd{^{2N}\phi}{x_1^{2N}}(\hat \kappa^*;\xi)(x_1-x^*_1)^{2N} \notag\\
&=\; \kappa^*_l + \frac{1}{(2N)!}\pd{^{2N}\phi}{x_1^{2N}}(\hat \kappa^*;\xi)(x_1-x^*_1)^{2N}, \; \;\;\;\;\text{where}\; \xi\in (x_1^*-\epsilon, x_1^*+\epsilon). \label{eq:tayloreven}
\end{align}
}Choose $\delta\in {\mathbb R}$ such that ${\rm sign}(\delta)={\rm sign}(\pd{^{2N}\phi}{x_1^{2N}}(\hat \kappa^*;\xi))$ and
 $\left(\delta (2N)!/\pd{^{2N}\phi}{x_1^{2N}}(\hat \kappa^*;\xi)\right)^{\frac{1}{2N}}<\epsilon$. For any
$\delta'\in (0, |\delta|)$, let
\begin{align*}
\tilde\delta'\;:=\; {\rm sign}(\delta)\delta'.
\end{align*}
By \eqref{eq:tayloreven}, we solve the equation $\phi(\hat \kappa^*;x_1)=\kappa_l^*+\tilde\delta'$, and we get two
distinct real solutions:
$$x_1^{(1)}=x^*_1+\left(\tilde\delta' (2N)!/\pd{^{2N}\phi}{x_1^{2N}}(\hat \kappa^*;\xi)\right)^{\frac{1}{2N}}, \;\text{and}\; x_1^{(2)}=x^*_1-\left(\tilde\delta' (2N)!/\frac{d^{2N} \phi}{d x_1^{2N}}(\hat\kappa^*;\xi)\right)^{\frac{1}{2N}}.$$
%Then by \eqref{eq:tayloreven},
%$x_1^{(1)}$ and $x_1^{(2)}$ are two distinct real solutions of the equation $\phi(\hat \kappa^*;x_1)=\kappa_l^*+\tilde\delta'$.
By \eqref{eq:pq} and \eqref{eq:phi},
$x_1^{(1)}$ and $x_1^{(2)}$ are real solutions of the equation $q(\tilde \kappa^*;x_1)=0$, where $\tilde \kappa^*:= (\kappa_1^*,\ldots,\kappa_{l-1}^*,\kappa_l^*+\tilde\delta',\kappa_{l+1}^*, \ldots, \kappa_m^*)$.
By \eqref{eq:tayloreven}, we have $\pd{\phi}{x_1}(\hat \kappa^*;x_1^{(i)})\neq 0$ for $i\in \{1, 2\}$. So, by
Lemma \ref{lm:phi}, both $x_1^{(1)}$ and $x_1^{(2)}$ are simple solutions of $q(\tilde \kappa^*;x_1)=0$.  %, where we set $\kappa^{**}_i=\kappa^*_i$ for all $i\neq 1$.
Clearly, by \eqref{eq:deriv-2}, the equation $\phi(\hat \kappa^*;x_1)=\kappa_l^*-\tilde\delta'$ has no real solutions for $x_1$, and so,  for $\kappa= (\kappa_1^*,\ldots,\kappa_{l-1}^*,\kappa_l^*-\tilde \delta',\kappa_{l+1}^*, \ldots, \kappa_m^*)$,
the equation $q(\kappa;x_1)=0$ has no real solutions.}
\end{proof}

\begin{lemma} \label{lm:odd}
Given a network $G$,
   suppose the stoichiometric subspace of $G$ is one-dimensional.
\textcolor{black}{
For a total-constant vector $c^*$, let $q(\kappa;x_1)$ be the polynomial defined in \eqref{eq:q}, let  $\ell$ be the index defined in \eqref{eq:ell} and let $\phi(\hat \kappa; x_1)$ be the rational function defined in \eqref{eq:phi}.
If for a rate-constant vector $\kappa^*$,
 $x_1^*$  is a simple solution, or a multiplicity-$(2N+1)$ ($N\in {\mathbb Z}_{>0}$) solution of
 $q(\kappa^*;  x_1)=0$, and if $\phi(\hat \kappa; x_1)$ is well-defined at $x_1^*$,
%If for a rate-constant vector $\kappa^*$ and a total-constant vector $c^*$,
%$x_1^*$ is a simple solution, or a  multiplicity-$(2N+1)$ ($N\in {\mathbb Z}_{>0}$) solution of
 %$g(x_1)=0$,
 then for every $\epsilon>0$,  there exists $\delta \in {\mathbb R}$ such that for all $\delta' \in (0,|\delta|)$,
 \begin{itemize}
 \item[(1)] for
 $\kappa= (\kappa_1^*,\ldots, \kappa^*_{\ell-1}, \kappa_{\ell}^*+{\rm sign}(\delta)\delta', \kappa_{\ell+1}^*, \ldots, \kappa_m^*)$, the equation $q(\kappa;x_1)=0$ has a simple real
solution $x_1^{(1)}$, for which $0<x_1^{(1)}-x_1^*<\epsilon$. % and $-\epsilon<x_1^{(2)}- x_1^*<0$, and
\item[(2)]
for $\kappa= (\kappa_1^*,\ldots, \kappa^*_{\ell-1}, \kappa_{\ell}^*-{\rm sign}(\delta)\delta', \kappa_{\ell+1}^*, \ldots, \kappa_m^*)$, the equation $q(\kappa;x_1)=0$ has a simple real solution $x_1^{(2)}$, for which $-\epsilon<x_1^{(2)}- x_1^*<0$.
\end{itemize}
}
\end{lemma}

\begin{proof}
\textcolor{black}{
%By Lemma \ref{lm:phi}, we have
%$\phi(\hat \kappa^*;x_1^*)=\kappa_l^*$, $\pd{\phi}{x_1}(\hat \kappa^*;x_1^*)=0$, $\ldots, \pd{^{2N}\phi}{x_1^{2N}}(\hat \kappa^*;x_1^*)=0$  and $\pd{^{2N+1}\phi}{x_1^{2N+1}}(\hat \kappa^*;x_1^*)\neq 0$.
Similar to the proof of Lemma \ref{lm:even}, for any $\epsilon>0$,
without loss of generality, we assume $\pd{^{2N+1}\phi}{x_1^{2N+1}}(\hat \kappa^*;x_1)\neq 0$ for any
$x_1\in (x_1^*-\epsilon, x_1^*+\epsilon)$.
By Lemma \ref{lm:phi},  the Taylor expansion of $\phi(\hat \kappa^*; x_1)$ over $(x_1^*-\epsilon, x_1^*+\epsilon)$ is
{\footnotesize
\begin{align}
\phi(\hat \kappa^*; x_1)\;&=\;\phi(\hat \kappa^*; x_1^*)+\sum_{k=1}^{2N}\frac{1}{k!}\pd{^{k}\phi}{x_1^k}(\hat \kappa^*;x_1^*)(x_1-x^*_1)^k
+\frac{1}{(2N+1)!}\pd{^{2N+1}\phi}{x_1^{2N+1}}(\hat \kappa^*;\xi)(x_1-x^*_1)^{2N+1} \notag\\
&=\; \kappa^*_l + \frac{1}{(2N+1)!}\pd{^{2N+1}\phi}{x_1^{2N+1}}(\hat \kappa^*;\xi)(x_1-x^*_1)^{2N+1}, \;\;\;\;\; \text{where}\; \xi\in (x_1^*-\epsilon, x_1^*+\epsilon). \label{eq:taylorodd}
\end{align}
}
Choose $\delta\in {\mathbb R}$ such that ${\rm sign}(\delta)={\rm sign}(\pd{^{2N+1}\phi}{x_1^{2N+1}}(\hat \kappa^*;\xi))$ and
 $$\left(\delta (2N+1)!/\pd{^{2N+1}\phi}{x_1^{2N+1}}(\hat \kappa^*;\xi)\right)^{\frac{1}{2N+1}}<\epsilon.$$ For any
$\delta'\in (0, |\delta|)$, let
\begin{align*}
\tilde\delta'\;:=\; {\rm sign}(\delta)\delta'.
\end{align*}
Then, for $\tilde \kappa^*= (\kappa_1^*,\ldots,\kappa_{l-1}^*,\kappa_l^*+\tilde\delta',\kappa_{l+1}^*, \ldots, \kappa_m^*)$,
the equation $q(\tilde \kappa^*;x_1)=0$ has a real solution
$x_1^{(1)}=x^*_1+\left(\tilde \delta' (2N+1)!/\pd{^{2N+1}\phi}{x_1^{2N+1}}(\hat \kappa^*;\xi)\right)^{\frac{1}{2N+1}}$.
By \eqref{eq:taylorodd}, we have $\pd{\phi}{x_1}(\hat \kappa^*;x_1^{(1)})\neq 0$. So, by Lemma \ref{lm:phi}, we know $\pd{q}{x_1}(\hat \kappa^*;x_1^{(1)})\neq 0$, and hence, $x_1^{(1)}$ is a  simple solution.
Similarly,  for $\kappa= (\kappa_1^*,\ldots,\kappa_{l-1}^*,\kappa_l^*-\tilde\delta',\kappa_{l+1}^*, \ldots, \kappa_m^*)$,
the equation $q(\kappa;x_1)=0$ has a simple real solution
$$x_1^{(2)}=x^*_1-\left(\tilde\delta' (2N+1)!/\pd{^{2N+1}\phi}{x_1^{2N+1}}(\hat \kappa^*;\xi)\right)^{\frac{1}{2N+1}}.$$
}
\end{proof}

\begin{example}\label{ex:evenodd}
We illustrate by Figure \ref{fig:main1} how Lemma \ref{lm:even} and Lemma \ref{lm:odd} work.
Consider the network $G$ in Example \ref{ex:main1}.
Choose
$c_1^*=-9$. Then, we can compute that the polynomial defined in \eqref{eq:q}  for $c^*=(c_1^*)$ is
$$q(\kappa; x_1) \;=\; \kappa_1 (-x_1+9)-\kappa_2 x_1+ \kappa_3 x_1^2 (-x_1+9).$$
By the computation in Example \ref{ex:main1}, we know the index $\tau$ defined in \eqref{eq:setk} is
$\tau=1$. Also, we know
$\gamma_{11}=0$, $\gamma_{12}=1$, and $\gamma_{13}=2$.
So,  the index $\ell$ defined in \eqref{eq:ell} is $1$.
So, the function defined in \eqref{eq:phi} is
$$\phi(\hat \kappa; x_1)\;=\;\frac{\kappa_2 x_1-\kappa_3 x_1^2 (-x_1+9)}{-x_1+9}.$$
We substitute $\hat \kappa = (\kappa_2^*, \kappa_3^*) = (16,\frac{3}{2})$ into $q(\kappa; x_1)$, and we get a
bivariate polynomial
$$ \mu(\kappa_{1}, x_1)\; :=\; \kappa_1 (-x_1+9)-16 x_1+ \frac{3}{2} x_1^2 (-x_1+9).$$
\begin{itemize}
\item[(a)]
\textcolor{black}{
For $\kappa^*:=(\kappa_1^*, \kappa_2^*, \kappa_3^*)=(\frac{1}{2},16,\frac{3}{2})$, the equation $q(\kappa^*;x_1)=0$ has a multiplic-ity-$2$ solution, which is $x^*_1\approx 0.6961103652$.
Also, $x^*_1$ is  a multiplicity-$2$ solution of $\mu(\kappa^*_1, x_1)=0$.
Clearly, $\phi(\hat \kappa; x_1)$ is well-defined at $x^*_1$.
We plot the point $(\kappa_{1}^*, x^*_1)$ and
 $\mu(\kappa_{1}, x_1)=0$ for \textcolor{black}{ $0\leq \kappa_{1}\leq 0.8$ and $0\leq x_1\leq 1.5$} in Figure \ref{fig:main1} (a).
 Choose $\delta =-\frac{1}{2}$. It is seen from Figure \ref{fig:main1} (a) that for any
 $\delta'\in (0, |\delta|)$, the equation $\mu(\kappa^*_{1}-\delta', x_1)=0$ has two distinct simple real solutions (e.g., see the two solid boxes for $\kappa_1=0.3$ in Figure \ref{fig:main1} (a)), which are sufficiently close to $x_1^*$, and the equation $\mu(\kappa^*_{1}+\delta', x_1)=0$ has no real solution.
}
\item[(b)]For $\kappa^*:=(\kappa_1^*, \kappa_2^*, \kappa_3^*)=(\frac{3}{10},16,\frac{3}{2})$, the equation $q(\kappa^*;x_1)=0$ has a simple solution, which is \textcolor{black}{$x^*_1\approx 1.201246094$}.
Clearly, $\phi(\hat \kappa; x_1)$ is well-defined at $x^*_1$.
We plot the point $(\kappa_{1}^*, x^*_1)$ and
 $\mu(\kappa_{1}, x_1)=0$ for \textcolor{black}{ $0\leq \kappa_{1}\leq 0.8$ and $0.8\leq x_1\leq 1.5$} in Figure \ref{fig:main1} (b).
 Choose $\delta =\frac{1}{4}$. It is seen from Figure \ref{fig:main1} (b) that for any
 $\delta'\in (0, \delta)$, the equation $\mu(\kappa^*_{1}+\delta', x_1)=0$ has a simple real solution (e.g., see the solid box at the intersection of the blue line $\kappa_1=0.5$ and the red curve in Figure \ref{fig:main1} (b)), which is sufficiently close to $x_1^*$, and similarly, the equation $\mu(\kappa^*_{1}-\delta', x_1)=0$ has a simple real solution too.
\end{itemize}
\end{example}
 % \subsection{Proofs of Theorem \ref{thm:nc} and Theorem \ref{thm:nondegmultistable}}
 \begin{figure}[h]%[H]
\centering
\subfigure[]{
\includegraphics[width=0.48\textwidth, height=0.48\textwidth]{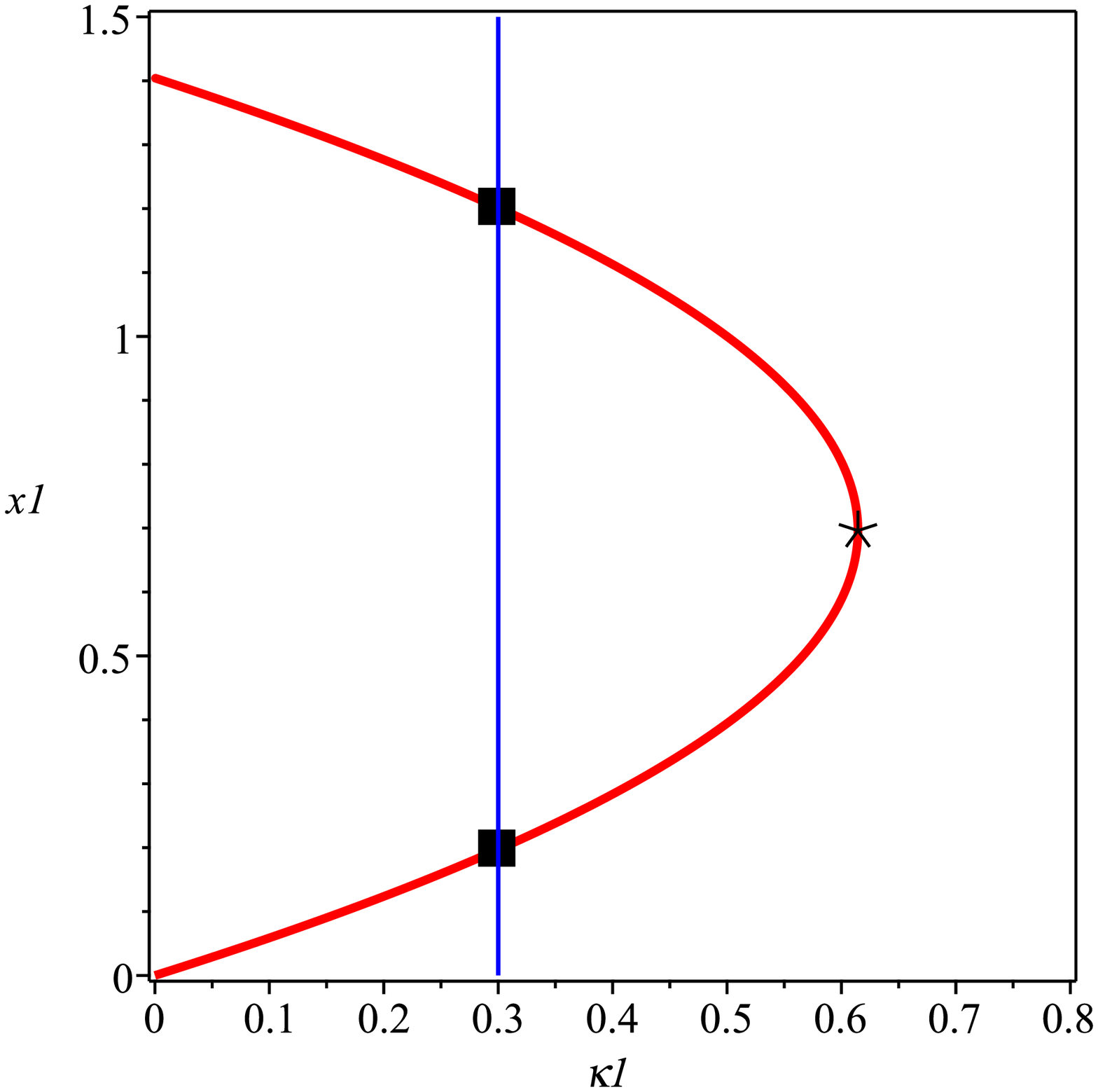}}
\subfigure[]{
\includegraphics[width=0.48\textwidth, height=0.48\textwidth]{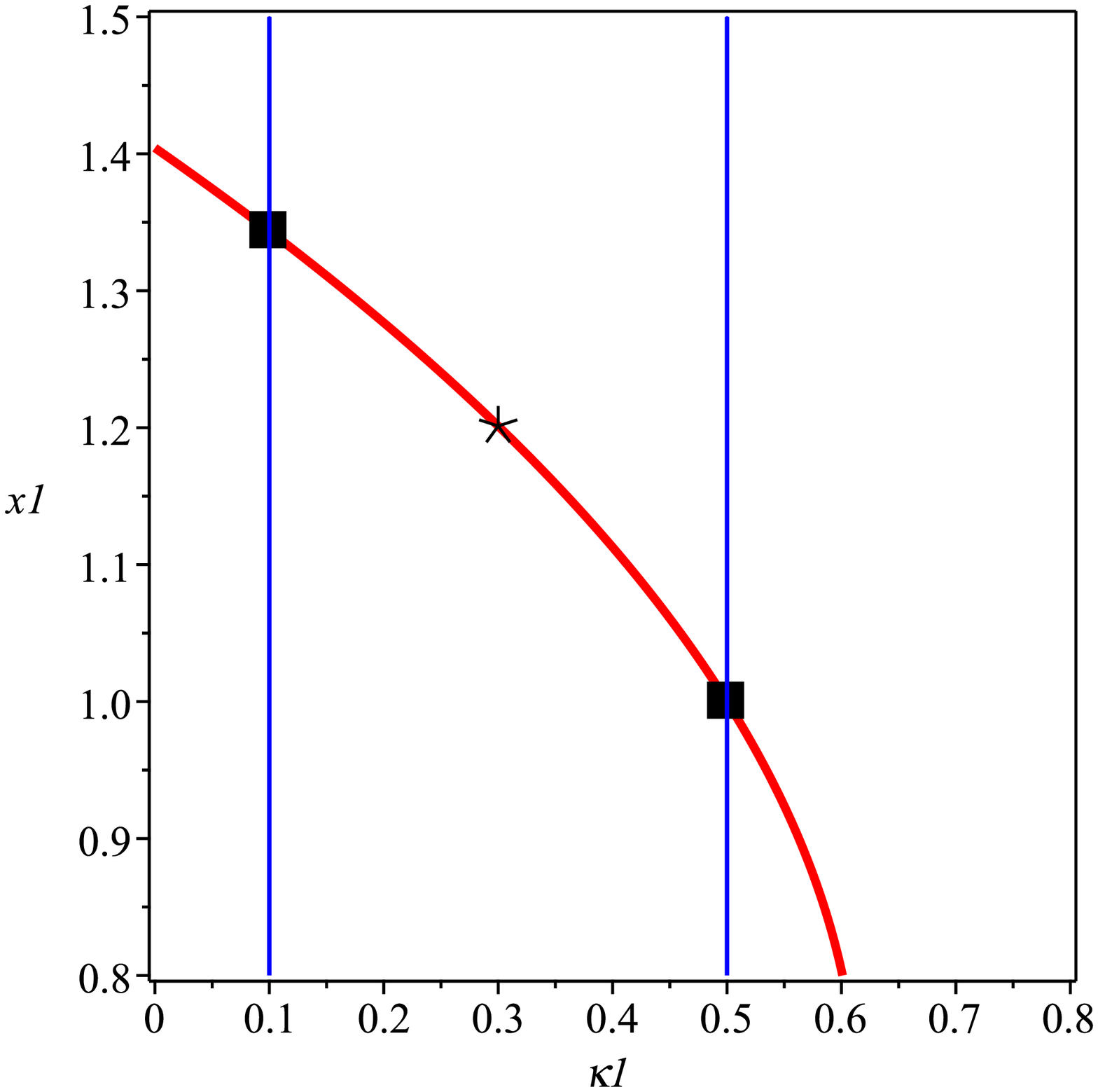}}
\caption{The two figures (a) and (b) illustrate Lemma \ref{lm:even}  and Lemma \ref{lm:odd}, respectively. See the corresponding examples in Example \ref{ex:evenodd}. These figures are generated by {\tt Maple2020}, see ``Plot.mw" in Table \ref{tab:sup}.}
\label{fig:main1}
\end{figure}

%\textcolor{blue}{...}\\
{\bf Proof of Theorem \ref{thm:nc}.}
\begin{proof}
We only need to show $cap_{pos}(G)\leq cap_{nondeg}(G)$.
 Suppose $cap_{pos}(G)=N$, i.e., there exist a rate-constant vector $\kappa^*$ and a total-constant vector $c^*$ such that $G$ has $N$ distinct positive steady states $x^{(1)},\ldots,x^{(N)}$ in ${\mathcal P}_{c^*}$.
 For the total-constant vector $c^*$, let $q(\kappa;x_1)$ be the polynomial defined in \eqref{eq:q}, let  $\ell$ be the index defined in \eqref{eq:ell} and let $\phi(\hat \kappa; x_1)$ be the rational function defined in \eqref{eq:phi}.
 By the proof of Lemma \ref{lm:qroot}, the first coordinates of those positive steady states \textcolor{black}{$x^{(1)}_1,\ldots,x^{(N)}_1$ are solutions of $q(\kappa^*;x_1)=0$ in the interval
 $\bigcap\limits_{i=1}^sI_i$, where $I_i$ is defined in \eqref{eq:iinterval}.}
 Note that $\bigcap\limits_{i=1}^sI_i\subset I$, where $I$ is defined in \eqref{eq:tinterval}. So, by Lemma \ref{lm:well-defined} (1),
 the function $\phi(\hat \kappa; x_1)$ is well-defined at $x^{(i)}_1$ for any $i\in \{1, \ldots, N\}$.
 Without loss of generality,
 assume the multiplicity of any solution in $\{x^{(1)}_1,\ldots,x^{(v)}_1\}$ $(v\leq N)$ is even,
 \textcolor{black}{and
 any solution in $\{x^{(v+1)}_1,\ldots,x^{(N)}_1\}$ is simple or has an odd multiplicity.
Denote the interval $\bigcap\limits_{i=1}^sI_i$  by $(b, E)$.  Here, we only prove the case when $E\neq +\infty$.
If $E=+\infty$, the argument  is  similar.
 Let
 $$\epsilon \;=\;\min\{\frac{\min \limits_{1\leq i\leq N}\{x^{(i)}_1\}-b}{2},\; \min\{\frac{|x^{(i)}_1-x^{(j)}_1|}{2}|1\leq i<j\leq N\},\; \frac{E-\max\limits_{1\leq i\leq N} \{x^{(i)}_1\}}{2}\}.$$
 By Lemma \ref{lm:even}, for $\epsilon>0$,  and for every $k\in \{1, \ldots, v\}$,
 there exists $\delta_k\in {\mathbb R}$ such that for any $\delta'\in (0, |\delta_k|)$,   $q(\kappa; x_1)=0$ has two distinct simple solutions in the
 interval $(x_1^{(k)}-\epsilon, x_1^{(k)}+\epsilon)$ for $$\kappa=\left(\kappa_1^*,\ldots, \kappa^*_{\ell-1}, \kappa_{\ell}^*+{\rm sign}(\delta_k)\delta', \kappa_{\ell+1}^*, \ldots, \kappa_m^*\right),$$ and
  $q(\kappa; x_1)=0$ has no real solutions in the
 interval $(x_1^{(k)}-\epsilon, x_1^{(k)}+\epsilon)$ for $$\kappa=\left(\kappa_1^*,\ldots, \kappa^*_{\ell-1}, \kappa_{\ell}^*-{\rm sign}(\delta_k)\delta', \kappa_{\ell+1}^*, \ldots, \kappa_m^*\right).$$ Note that there exist  %at least
 $\lceil \frac{v}{2} \rceil$ indices $k_1, \ldots, k_{\lceil \frac{v}{2} \rceil}$ in $\{1, \ldots, v\}$ such that
 $\delta_{k_1}, \ldots, \delta_{k_{\lceil \frac{v}{2} \rceil}}$ have the same sign.
 Without loss of generality, assume $\delta_{k_1}>0$.
By  Lemma \ref{lm:odd}, for $\epsilon>0$,  and for every $j\in \{v+1, \ldots, N\}$,
 there exists $\delta_j>0$ such that for any $\delta'\in (0, \delta_j)$,   $q(\kappa;x_1)=0$ has a simple solution in the
 interval $(x_1^{(j)}-\epsilon, x_1^{(j)}+\epsilon)$ for $\kappa=\left(\kappa^*_1,\ldots,\kappa^*_{\ell-1}, \kappa^*_{\ell}\pm \delta', \kappa^*_{\ell+1},\ldots,\kappa^*_m\right)$. Let $\delta=\min \{|\delta_1|, \ldots, |\delta_v|, \delta_{v+1}, \ldots, \delta_N\}$.
 Therefore, for any $\delta'\in (0, \delta)$, for  $\kappa=\left(\kappa^*_1,\ldots,\kappa^*_{\ell-1}, \kappa^*_{\ell}+ \delta', \kappa^*_{\ell+1},\ldots,\kappa^*_m\right)$, $q(\kappa; x_1)=0$ has
 at least $2\lceil \frac{v}{2} \rceil + N-v$ distinct simple solutions in $(b, E)$. By Lemma \ref{lm:qroot}, the network $G$ admits at least $2\lceil \frac{v}{2} \rceil + N-v$  positive steady states. By Lemma \ref{lm:jactildeg} and Lemma \ref{lm:nonde}, these $2\lceil \frac{v}{2} \rceil + N-v$  steady states are nondegenerate. So, we have $cap_{pos}(G)=N\leq 2\lceil \frac{v}{2} \rceil + N-v \leq cap_{nondeg}(G)$.}
 %(Remark here, the above argument shows that the number $d$ must be  even  since we know $cap_{nondeg}(G)\leq cap_{pos}(G)=N$.)
 \end{proof}

\begin{corollary}\label{cry:ww'}
 \textcolor{black}{
Given a network $G$ \eqref{eq:network} with a one-dimensional stoichiometric subspace}, suppose $cap_{pos}(G)<+\infty$.
 Assume ${\mathcal W}$ and ${\mathcal B}$ are the two sets of parameters defined in \eqref{eq:witness} and \eqref{eq:bwd}.
 Define a subset ${\mathcal W}_{nondeg}$ of ${\mathcal W}$:
 {\footnotesize
 \begin{align}
 {\mathcal W}_{nondeg}\;&:=\; \{(\kappa^*, c^*)\in {\mathcal W}|\text{for}\; (\kappa^*, c^*),\; \text{all positive steady states of}\; G \; \text{are nondegenerate}\}. \label{eq:witness'}
 \end{align}}Then, ${\mathcal W}\cap{\mathcal B}\neq \emptyset$ if and only if ${\mathcal W}_{nondeg}\cap{\mathcal B}\neq \emptyset$.
 \end{corollary}

 \begin{proof}
 Since ${\mathcal W}_{nondeg}\subset {\mathcal W}$, we only need to prove that
 if ${\mathcal W}\cap{\mathcal B}\neq \emptyset$, then ${\mathcal W}_{nondeg}\cap{\mathcal B}\neq \emptyset$.
 Suppose $(\kappa^*,c^*)\in {\mathcal W}\cap{\mathcal B}$. By the definition of ${\mathcal W}$ (see \eqref{eq:witness}), for $\kappa^*$, $G$ has $cap_{pos}(G)$ positive steady states in ${\mathcal P}_{c^*}$. If $(\kappa^*,c^*)\not\in {\mathcal W}_{nondeg}$, then
  by the proof of Theorem \ref{thm:nc}, there exist an index $\ell\in \{1, \ldots, s\}$ and a positive number $\delta$ such that for any $\delta'\in (0,\delta)$, for  $\kappa=\left(\kappa^*_1,\ldots,\kappa^*_{\ell-1}, \kappa^*_{\ell}+ \delta', \kappa^*_{\ell+1},\ldots,\kappa^*_m\right)$ or  $\kappa=\left(\kappa^*_1,\ldots,\kappa^*_{\ell-1}, \kappa^*_{\ell}- \delta', \kappa^*_{\ell+1},\ldots,\kappa^*_m\right)$, $G$
 has $cap_{pos}(G)$ nondegenerate positive steady states in ${\mathcal P}_{c^*}$. Notice that the left-hand side of the inequality in \eqref{eq:bwd} is continuous with respect to $\kappa^*$. Thus, there exists $\delta''\in (0,\delta)$ such that
 $(\kappa^{**}, c^*)\in {\mathcal B}$, where
 $\kappa^{**}:=\left(\kappa^*_1,\ldots,\kappa^*_{\ell-1}, \kappa^*_{\ell}\pm \delta'', \kappa^*_{\ell+1},\ldots,\kappa^*_m\right)$. Therefore, we have ${\mathcal W}_{nondeg}\cap{\mathcal B}\neq \emptyset$.
 \end{proof}

\section{Multistability}\label{sec:mstab}
The goal of this section is to prove Theorem \ref{thm:bdm} (the main result), Corollary \ref{cry:bdm}, and Corollary \ref{cry:bd}.
 From the proof presented later, we will see that the part (a) of Theorem \ref{thm:bdm} follows from the results in the previous two sections. So, the main task is to prove the part (b).
 The most crucial step is to show $q(\kappa^*;a)\neq 0$ for some $(\kappa^*, c^*)$ (see Lemma \ref{lm:nobd}), where the polynomial $q$ is defined in \eqref{eq:q}, and the number $a$ is the left endpoint of the interval $I$ defined in \eqref{eq:tinterval}. After this step, we can determine the number of stable positive steady states by the sign of $q(\kappa^*;a)$ (see Lemma \ref{lm:bdm} and Corollary \ref{cry:lmbdm}). From the proof, we will see
 the left-hand side of the inequality in \eqref{eq:bwd} is an explicit expression of $q(\kappa^*;a)$ in terms of $(\kappa^*, c^*)$. We first prepare a list of lemmas.
\begin{lemma}\label{lm:lmtinterval}
Given a network $G$,
   suppose the stoichiometric subspace of $G$ is one-dimensional.
   %Let  $g(\kappa; c; x_1)$ be the polynomial defined in \eqref{eq:g}.
   %For a total-constant vector $c^*$, let $I$ be the open interval defined in \eqref{eq:tinterval}.
    For a total-constant vector $c^*$, let $I$ be the open interval defined in \eqref{eq:tinterval}, and
  let  $q(\kappa; x_1)$ be the polynomial defined in \eqref{eq:q}.
   If  for a rate-constant vector $\kappa^*$,
   $q(\kappa^*;x_1)=0$ has  $N$ different solutions in $I$, then
   there exists another total-constant vector $\tilde c^*$ such that
   the equation  $\tilde q(\kappa^*;x_1)=0$ corresponding to $\tilde c^*$ has
   at least $N$ different solutions in the interval $\cap_{j=1}^s \tilde I_j$, where
   $\tilde I_j$ is the interval defined in \eqref{eq:iinterval} for the total-constant vector $\tilde c^*$.
\end{lemma}
\begin{proof}
Assume that
$q(\kappa^*;x_1)=0$ has  $N$ different solutions $x_1^{(1)},\ldots, x_1^{(N)}$ in $I$ for some
rate-constant vector $\kappa^*$.
For any $j\in \{1, \ldots, s\}$, let $A_j$ and $B_j$ be the numbers defined in \eqref{eq:a} and \eqref{eq:b} for $c^*$.
Let ${\mathcal H}$ be the set of indices defined in \eqref{eq:nindexi}
for  $c^*$.
Since ${\mathcal H}$ is finite, we can choose a positive number $\rho$ such that for any $k\in {\mathcal H}$, $\rho > \frac{B_k}{A_k}$.
For any $j\in \{2, \ldots, s\}$,
 suppose
$j\in [k]_{c^*}$, where $k\in \{1, \ldots, r\}$ (recall the definition of $[k]_{c^*}$ in \eqref{eq:iec},
and recall that $r$ denotes the number of equivalence classes).
 We define the $(j-1)$-coordinate of a new total-constant vector $\tilde c^*$:
%any $k\in \{1, \ldots, r\}$, and for every $j\in [k]$, define
%\textcolor{blue}
%{
%\begin{align*}
%\tilde c^*_{j-1}\;=\;
%\begin{cases}
%c^*_{j-1}, &\text{if}\; j\in {\mathcal H}, \;\text{or}\; A_j =0\\
%0, & \text{if}\; j\not\in {\mathcal H}, \;\text{and}\; A_j >0\\
%\min \limits_{1\leq i\leq 4} \{(\beta_{j1}-\alpha_{j1})x_1^{(i)}\} -1,  & \text{if}\; j\not\in {\mathcal H},  A_j<0, %\;\text{and}\; \beta_{11}-\alpha_{11}>0 \\
%\max  \limits_{1\leq i\leq 4} \{(\beta_{j1}-\alpha_{j1})x_1^{(i)}\} +1, & \text{if}\; j\not\in {\mathcal H},  A_j<0, %\;\text{and}\; \beta_{11}-\alpha_{11}<0
%\end{cases}.
%\end{align*}
%}
\begin{align}\label{eq:star}
\tilde c^*_{j-1}\;:=\;
\begin{cases}
c^*_{j-1}, &\text{if}\; k\in {\mathcal H}, \;\text{or}\; A_j =0\\
-\rho(\beta_{j1}-\alpha_{j1}), & \text{if}\; k\not\in {\mathcal H}, \;\text{and}\; A_j >0\\
(\beta_{j1}-\alpha_{j1})(\max  \limits_{1\leq i\leq (N+1)} \{x_1^{(i)}\} +\rho), & \text{if}\; k\not\in {\mathcal H},\;\text{and}\;  A_j<0. %\;\text{and}\; \beta_{11}-\alpha_{11}<0
\end{cases}
\end{align}
%And, for every $i$ $\in \{1, \ldots, N+1\}$, we define a vector $y^{(i)}\in {\mathbb R}^s$ such that
%$y^{(i)}_1=x^{(i)}_1$ and $y^{(i)}_j=A_jx^{(i)}_j-\frac{\tilde c^*_{j-1}}{\beta_{11}-\alpha_{11}} $ for all
%$j\in \{2, \ldots, s\}$.
%\textcolor{blue}{..}

Now, for  $\tilde c^*$, let $\tilde {\mathcal H}$ be the set of indices defined in \eqref{eq:nindexi}.
By \eqref{eq:star}, it is easy to verify that
\begin{itemize}
%\item[(I)] $[k]_{\tilde c^*} = [k]_{c^*}$, for any $k\in {\mathcal H}$,
\item[(I)] $\tilde {\mathcal H}={\mathcal H}$, and
\item[(II)] $\tilde q(\kappa^*;x_1)=q(\kappa^*;x_1)$.
\end{itemize}
 For any $k\in \{1, \ldots, s\}$,
%let $[k]$ and $\tilde{[k]}$ be the equivalence classes defined in \eqref{eq:iec}  for $c^*$ and $\tilde c^*$ respectively,
%and
%let $Y_{k}(x_1)$ and $\tilde Y_{k}(x_1)$ be the functions defined
%in \eqref{eq:y}   for $c^*$ and $\tilde c^*$ respectively.
%It is easy to see that for any $k\in {\mathcal H}$, we have
%$[k]_{c^*}=[k]_{\tilde c^*}$ and
%$Y_{k}(x_1)=\tilde Y_{k}(x_1)$.
%So, by \eqref{eq:q}, we have $\tilde q(\kappa^*;x_1)=q(\kappa^*;x_1)$.
So, by (II), we know
$x_1^{(1)}, \ldots, x_1^{(N)}$ are also solutions to $\tilde q(\kappa^*;x_1)=0$.

 Below, we show that for every $i\in \{1, \ldots, N\}$, we have
$x^{(i)}_1\in \bigcap\limits_{j=1}^s\tilde I_{j}$.
In fact, note that $\{1, \ldots, s\}=(\bigcup\limits_{k\in \tilde{\mathcal H}}[k]_{\tilde c^*})\cup (\bigcup\limits_{k\not\in \tilde {\mathcal H}}[k]_{\tilde c^*})$.
We know
$x^{(i)}_1\in I = \bigcap\limits_{k\in {\mathcal H}}I_{k}$.
 For every $k\in {\mathcal H}$, by \eqref{eq:star}, we have $I_{k}=\tilde I_{k}$.
 So, by (I), for every $k\in \tilde {\mathcal H}$,  $x^{(i)}_1 \in \tilde I_{k}$,
 and hence, for any $j\in [k]_{\tilde c^*}$, $x^{(i)}_1 \in \tilde I_{j}$
since by \eqref{eq:iec} and \eqref{eq:iinterval}, $\tilde I_k = \tilde I_j$.
On the other hand, for any $k\not\in \tilde {\mathcal H}$, and for every $j\in [k]_{\tilde c^*}$,
let $\tilde A_j$ and $\tilde B_j$ be the numbers defined in \eqref{eq:a} and \eqref{eq:b} for $\tilde c^*$. By \eqref{eq:a}, we see that $A_j=\tilde A_j$.
By \eqref{eq:star},
if $A_j>0$, then
we have $x^{(i)}_1>-\frac{\tilde B_j}{A_j}=-\rho$ (since $x^{(i)}_1>0$), and
if $A_j<0$, then
we have $x^{(i)}_1<-\frac{\tilde B_j}{A_j}=\max  \limits_{1\leq i\leq N} \{x_1^{(i)}\} +\rho$.
Notice that the left endpoint of $I$ is no less than $0$ (recall Remark \ref{rmk:end1}). So, we have
$x^{(i)}_1>0$.
So, by \eqref{eq:iinterval},  we always have $x^{(i)}_1\in \tilde I_{j}$. Therefore, we know that $\tilde q(\kappa^*;x_1)=0$ has at least $N$ different solutions in the interval $\bigcap\limits_{j=1}^s\tilde I_{j}$.
\end{proof}

%\begin{example}\label{ex:lmtinterval}
%\textcolor{blue}{not sure if we should put an example here
% Consider the network $G$
%$$ 2X_1 + 2X_2 + X_3 +X_4 \xrightarrow{\kappa_1} 3X_1 + X_2,\;\;
% X_1+2X_3+X_4 \xrightarrow{\kappa_2} X_2 + 3X_3+2X_4. $$
% }
% \end{example}

\begin{lemma}\label{lm:tinterval}
Given a network $G$,
   suppose the stoichiometric subspace of $G$ is one-dimensional.
   %Let  $g(\kappa; c; x_1)$ be the polynomial defined in \eqref{eq:g}.
   %For a total-constant vector $c^*$, let $I$ be the open interval defined in \eqref{eq:tinterval}.
    For a total-constant vector $c^*$, let $I$ be the open interval defined in \eqref{eq:tinterval}, and
  let  $q(\kappa; x_1)$ be the polynomial defined in \eqref{eq:q}.
   If $cap_{pos}(G)=N$ $(0\le N<+\infty)$, then for any rate-constant vector $\kappa^*$,
   $q(\kappa^*;x_1)=0$ has at most $N$ different solutions in $I$.
      %and if
 %for a rate-constant vector $\kappa^*$, $G$ has N nondegenerate positive
 %steady states in the stoichiometric compatibility class ${\mathcal P}_{c^*}$,
 %then  %for the rate-constant vector $\kappa^*$ and the total-constant vector $c^*$,
 %$q(\kappa^*;x_1)=0$ has at most $N$ simple solutions in $I$.
 %the open interval $I$ defined in \eqref{eq:tinterval}.
%${\mathcal H}\;:= \{i\in \{2, \ldots, s\}| \alpha_{i1}, \ldots, \alpha_{im}\;\text{are not all the same}\}\;$.
\end{lemma}
\begin{proof}
%Suppose for the rate-constant vector $\kappa^*$ and the total-constant vector $c^*$, the $N$ nondegenerate positive steady states are $x^{(1)},\ldots, x^{(N)}$. \textcolor{blue}{By \eqref{eq:sg} and \eqref{eq:q}},
%their first coordinates are solutions to the equation $q(\kappa^*;x_1)=0$.
% By \eqref{eq:h}, for every $x^{(i)}$ and for any $j$ $(1\leq j\leq s)$,
%$x^{(i)}_j\;=\;A_jx^{(i)}_1+B_j>0.$
%So $x^{(i)}_1$ is contained in the interval $\bigcap\limits_{i=1}^sI_i\subset I$, where  $I_i$ and $I$ are defined in \eqref{eq:tinterval}, and hence,
%by Lemma \ref{lm:jac}, Lemma \ref{lm:nonde}, and Lemma \ref{lm:jactildeg},
%$q(\kappa^*;x_1)=0$ has at least $N$  simple solutions $x_1^{(1)},\ldots,x_1^{(N)}$ in the open interval $I$.

%In the rest of the proof, we prove that $q(\kappa^*;x_1)=0$ has no other solution in $I$. In fact,
If $q(\kappa^*;x_1)=0$ has at least $N+1$ different  solutions in $I$, then by Lemma \ref{lm:lmtinterval},
   there exists a total-constant vector $\tilde c^*$ such that
   the equation  $\tilde q(\kappa^*;x_1)=0$ corresponding to $\tilde c^*$ has
   at least $N+1$ different solutions in the interval $\cap_{j=1}^s \tilde I_j$, where
   $\tilde I_j$ is the interval defined in \eqref{eq:iinterval} for the total-constant vector $\tilde c^*$.
Thus, by the proof of Lemma \ref{lm:qroot}, for $\kappa = \kappa^*$,
the network $G$ has at least $N+1$ positive steady states in ${\mathcal P}_{\tilde c^*}$, %$y^{(i)}$ $(i=1, \ldots, N+1)$,
which is a contradiction to the hypothesis that
$cap_{pos}(G)=N$.
\end{proof}

The following definition for a degenerate steady state with a multiplicity is motivated by Lemma \ref{lm:nonde}, Lemma \ref{lm:qroot} and Lemma \ref{lm:jactildeg}.

\begin{definition}\label{def:multi}
   \textcolor{black}{
  Given a network $G$ \eqref{eq:network}, suppose the stoichiometric subspace of $G$ is one-dimensional. For a total-constant vector $c^*$, %let $I:=(a,M)$ be the open interval defined in \eqref{eq:tinterval}, and
  let $q(\kappa; x_1)$ be the polynomial defined in \eqref{eq:q}. Suppose for a rate-constant vector $\kappa^*$,  the network $G$ has a positive steady state $x^*$  in ${\mathcal P}_{c^*}$.
   If $x^*_1$ is a multiplicity-$N$ ($N\geq 2$) solution of $q(\kappa^*; x_1)=0$, then
    we say $x^*$ is \defword{degenerate with multiplicity-$N$}.}
\end{definition}

\begin{lemma}\label{lm:nobd}
Given a network $G$, suppose the stoichiometric subspace of $G$ is one-dimensional.
  For a total-constant vector $c^*$, let $I:=(a,M)$ be the open interval defined in \eqref{eq:tinterval}, and
  let  $q(\kappa; x_1)$ be the polynomial defined in \eqref{eq:q}.
   If $cap_{pos}(G)=N$ $(0\leq N<+\infty)$, and if
 for a rate-constant vector $\kappa^*$, $G$ has $N$ \textcolor{black}{positive steady states} in the stoichiometric compatibility class ${\mathcal P}_{c^*}$, \textcolor{black}{and any degenerate positive steady state has an odd multiplicity,} then %for $\kappa=\kappa^*$ and for $c=c^*$,
 we have
 $q(\kappa^*;a)\neq 0$.
 %, where
 %$a$ is the left endpoint of the open interval $I$ defined in \eqref{eq:interval}.
\end{lemma}

\begin{proof}
Suppose for  the total-constant vector $c^*$ and for the rate-constant vector $\kappa^*$, the $N$ positive steady states are $x^{(1)},\ldots,x^{(N)}$ such that
 $x_1^{(1)}<\ldots<x_1^{(N)}$. By the proof of Lemma \ref{lm:qroot}, for any $i\in  \{1,\ldots,N\}$, $x_1^{(i)}$ is a solution of $q(\kappa^*; x_1)=0$ in $\bigcap \limits_{j=1}^sI_j$, where $I_j$ is defined in \eqref{eq:iinterval}. By \eqref{eq:tinterval},  $\bigcap \limits_{j=1}^sI_j\subset I$. Below, we prove the conclusion by contradiction.
 Assume that $q(\kappa^*;a)=0$. %Then, by \eqref{eq:sg}, we have $g(\kappa^*;c^*;a)=0$.
 Let
 $$\epsilon\;=\;\frac{1}{2}\min\;\{\;|a-x_1^{(1)}|,\;\min\limits_{1\leq i\leq N-1}\{|x_1^{(i)}-x_1^{(i+1)}|\},\;|x_1^{(N)}-M|\}\;.$$
 For the total-constant vector $c^*$, let $\ell$ be the index defined in \eqref{eq:ell}, and let $\phi(\hat \kappa; x_1)$ be the function defined in \eqref{eq:phi}.
 By Lemma \ref{lm:well-defined} (1) and (3), $\phi(\hat \kappa; x_1)$ is well-defined at
 $x^{(i)}_1$ for any $i\in \{1, \ldots, N\}$, and at $a$.
By Lemma \ref{lm:jactildeg} and Lemma \ref{lm:nonde}, $x^{(i)}_1$ is a simple, or \textcolor{black}{a multiplicity-$(2K+1)$ solution} of $q(\kappa^*;x_1)=0$ (here, $K\in {\mathbb Z}_{> 0}$). So,
 by Lemma \ref{lm:odd}, for every $i\in \{1, \ldots, N\}$, for $\epsilon>0$, there exists $\delta_i>0$ such that for all $\delta' \in (0, \delta_i)$,
for
 $\kappa= (\kappa_1^*,\ldots, \kappa^*_{\ell-1}, \kappa_{\ell}^*\pm\delta', \kappa_{\ell+1}^*, \ldots, \kappa_m^*)$, the equation $q(\kappa; x_1)=0$ has a simple real
solution $y_1^{(i)}$, for which $|y_1^{(i)}-x_1^{(i)}|<\epsilon$.

 Either if $a$ is a multiplicity-$2K$ ($K\in {\mathbb Z}_{>0}$) solution of $q(\kappa^*;x_1)=0$, or if  $a$ is a simple, or a multiplicity-$(2K+1)$ ($K\in {\mathbb Z}_{>0}$) solution of  $q(\kappa^*;x_1)=0$, then by
 Lemma \ref{lm:even} (1) or Lemma \ref{lm:odd} (1), for $\epsilon>0$,
 there exists $\delta_0 \in {\mathbb R}$ such that for all $\delta' \in (0,|\delta_0|)$, for $\kappa= (\kappa_1^*,\ldots, \kappa^*_{\ell-1},\kappa_{\ell}^*+{\rm sign}(\delta_0)\delta', \kappa_{\ell+1}^*, \ldots, \kappa_m^*)$, the equation $q(\kappa; x_1)=0$ has a simple  solution $\tilde a$ such that $0<\tilde a-a<\epsilon$ (i.e., $a<\tilde a<a+\epsilon$).
 Let $\delta=\min\limits_{0\leq i\leq N}\{|\delta_i|\}$.
 We choose any $\delta'\in (0, \delta)$.
 Then, for
 $\kappa= (\kappa_1^*,\ldots, \kappa^*_{\ell-1},\kappa_{\ell}^*+{\rm sign}(\delta_0)\delta', \kappa_{\ell+1}^*, \ldots, \kappa_m^*)$,
 $q(\kappa;x_1)=0$ has at least $N+1$ different solutions ($y_1^{(1)}, \ldots, y_1^{(N)}$ and $\tilde a$)
 in the interval
 $I=(a,M)$. \textcolor{black}{This cannot happen by Lemma \ref{lm:tinterval}}.
 %, and so, the network $G$ admits $4$ nondegenerate positive steady states,  which contradicts to $cap_{pos}(G)=3$.
 Therefore,
 we have $q(\kappa^*;  a)\neq 0$.
\end{proof}

\begin{remark}\label{rmk:nobd}
If the right endpoint of the interval $I=(a, M)$ stated in Lemma \ref{lm:nobd} is a finite number, then we can also add $q(\kappa^*; M)\neq 0$ into the conclusion of Lemma \ref{lm:nobd}.
 By Lemma \ref{lm:uni} and Lemma \ref{lm:signend},  if all real solutions of $q(\kappa^*;x_1)=0$ are simple, which means all positive steady states are nondegenerate, then
the sign of $q(\kappa^*;a)$ gives the Brouwer degree (i.e., the summation of the signs of derivatives at all solutions $\sum\limits_{i=1}^N {\tt sign}(\frac{d q}{d x_1}(\kappa^*; x^{(i)}_1))$, see \cite[Theorem 1]{CFMW}) of the univariate function $q(\kappa^*; x_1)$
over a bounded subset of  $(a, M)$.
\end{remark}

\begin{lemma}\label{lm:x1-nondeg}
 \textcolor{black}{
 Given a network $G$ \eqref{eq:network} with a one-dimensional stoichiometric subspace, suppose $cap_{pos}(G)=2N+1$ $(0\leq N<+\infty)$.
 If
 for a rate-constant vector $\kappa^*$ and a total-constant vector $c^*$, $G$ has exactly $2N+1$ positive steady states $x^{(1)}, \ldots, x^{(2N+1)}$ $(x_1^{(1)}<\ldots<x_1^{(2N+1)})$ in the stoichiometric compatibility class ${\mathcal P}_{c^*}$,
 %and suppose the positive steady states are $x^{(1)}, \ldots, x^{(2N+1)}$ such that
 %$x_1^{(1)}<\ldots<x_1^{(2N+1)}$,
 and if $N+1$ of these positive steady states are stable, then
\begin{itemize}
\item[(a)]$x^{(1)}, x^{(3)},\ldots,$ $x^{(2N+1)}$ are stable, and
\item[(b)]any degenerate positive steady state has an odd multiplicity.
\end{itemize}
 }
\end{lemma}

\begin{proof}
 (a) %If $N+1$ of these positive steady states $x^{(1)}, \ldots, x^{(2N+1)}$ are stable, then by Corollary \ref{cry:sign-stable} (a), $x^{(1)}, x^{(3)},\ldots,$ $x^{(2N+1)}$ are the stable steady states, and thus, we have the conclusion.
 The conclusion follows from Corollary \ref{cry:sign-stable} (a).
 (b) \textcolor{black}{For the total-constant vector $c^*\in {\mathbb R}^{s-1}$,
let $q(\kappa;x_1)$ be the polynomial defined in \eqref{eq:q}.
By the proof of Lemma \ref{lm:qroot}, $x_1^{(1)}, \ldots, x_1^{(2N+1)}$ are real solutions of the equation $q(\kappa^*;x_1)=0$ in an open interval, and the equation has no other real solutions in this interval.
 So, for any $i\in \{1, \ldots, N\}$, by (a) and Lemma \ref{lm:jactildeg}, we have
 $$\pd{q(\kappa^*; x_1)}{x_1}(x^{(2i-1)}_1)\pd{q(\kappa^*; x_1)}{x_1}(x_1^{(2i+1)})>0.$$
 Thus, by Lemma \ref{lm:uni}, for any $i=1,\dots,N$, if $x^{(2i)}$ is degenerate, then it has an odd multiplicity. Then, the conclusion follows from (a).}
\end{proof}

\begin{lemma}\label{lm:bdm}
Given a network $G$ \eqref{eq:network} with a one-dimensional stoichiometric subspace, suppose $cap_{pos}(G)=2N+1$ $(0\leq N<+\infty)$.
 If
%For a total-constant vector $c^*$,
%let
%${\mathcal H}$ be the set of indices defined in \eqref{eq:nindexi}.
 for a rate-constant vector $\kappa^*$ \textcolor{black}{and a total-constant vector $c^*$}, $G$ has $2N+1$ positive steady states in the stoichiometric compatibility class ${\mathcal P}_{c^*}$, \textcolor{black}{and if $N+1$ of these positive steady states are stable, then}
 \begin{align}\label{eq:bwd'}
 \sum\limits_{j\in \mathcal L}(\beta_{1j}-\alpha_{1j})\kappa^*_{j}\prod\limits_{i\in {\mathcal J}}|A_i|^{\alpha_{ij}}\prod\limits_{i\in \{1, \ldots, s\}\backslash{\mathcal J}}B_i^{\alpha_{ij}} \prod \limits_{k\in {\mathcal H}, k\neq \tau}(\frac{B_k}{|A_k|}-\frac{A_k}{|A_k|}\frac{B_{\tau}}{A_{\tau}})^{\gamma_{kj}}>0.
 \end{align}
 where for the total constant-vector $c^*$, the notions $A_i$, $B_i$, $\gamma_{kj}$, ${\mathcal J}$, ${\mathcal H}$, $\tau$, and ${\mathcal L}$ are defined in \eqref{eq:a}--\eqref{eq:L}.
 %where
%  \begin{align}
% A_j \;\text{and}\; B_j \; \text{are defined in} \; \eqref{eq:a} \;\text{and}\; \eqref{eq:b},
% \end{align}
 %\begin{align}\label{eq:mathcalC}
 %{\mathcal C}_j \;=\; \prod\limits_{i\in {\mathcal J}}|A_i|^{\alpha_{ij}}\prod\limits_{i\in \{1, \ldots, s\}\backslash{\mathcal J}}B_i^{\alpha_{ij}}.
 %\end{align}
% $$\tau \;\text{is the index defined in \eqref{eq:setk}, and}$$
\end{lemma}

%{\bf Proof of Theorem \ref{thm:bdm}.}
\begin{proof}
 Suppose the positive steady states are $x^{(1)}, \ldots, x^{(2N+1)}$ such that
 $x_1^{(1)}<\ldots<x_1^{(2N+1)}$. Then by Lemma \ref{lm:x1-nondeg}, $x^{(1)}$ is stable and any degenerate steady state has an odd multiplicity. For the total-constant vector $c^*$, let $q(\kappa;x_1)$ be the polynomial defined in \eqref{eq:q}.
  By the proof of Lemma \ref{lm:qroot}, and by Lemma \ref{lm:jactildeg}, $x^{(1)}_1$ is a simple solution of $q(\kappa^*;x_1)=0$.
 Note that
   the left endpoint of the interval $I$ defined in \eqref{eq:tinterval} is $a=-\frac{B_{\tau}}{A_{\tau}}$.
 By Lemma \ref{lm:nobd},  we have $q(\kappa^*;a)\neq 0$.
   Thus, by Lemma \ref{lm:signend}, we have $${\tt sign}(\pd{q}{x_1}(\kappa^*;x_1^{(1)}))=-{\tt sign}(q(\kappa^*;a)).$$
By Lemma \ref{lm:uni}, for every $i\in \{2, \ldots, 2N+1\}$,
\begin{align}\label{eq:lmbdm}
{\tt sign}(\pd{q}{x_1}(\kappa^*;x_1^{(i)})){\tt sign}(\pd{q}{x_1}(\kappa^*;x_1^{(i-1)}))\leq 0.
\end{align}
So, by
Lemma \ref{lm:jactildeg}, if
there exist $N+1$ stable positive steady states,
then $q(\kappa^*;a)>0$.
Therefore, the condition \eqref{eq:bwd'} follows from the fact that
$$q(\kappa^*;a)\;=\;q(\kappa^*;-\frac{B_{\tau}}{A_{\tau}})\;=\; (\beta_{11}-\alpha_{11})\sum\limits_{j\in \mathcal L}\lambda_j\kappa^*_{j}{\mathcal C}_j \prod \limits_{k\in {\mathcal H}, k\neq \tau}(\frac{B_k}{|A_k|}-\frac{A_k}{|A_k|}\frac{B_{\tau}}{A_{\tau}})^{\gamma_{kj}},$$
and the equality $(\beta_{11}-\alpha_{11})\lambda_j =\beta_{1j}-\alpha_{1j}$ (see \eqref{eq:scalar}).
\end{proof}

\begin{corollary}\label{cry:lmbdm}
Given a network $G$ \eqref{eq:network} with a one-dimensional stoichiometric subspace, suppose $cap_{pos}(G)=2N+1$ $(0\leq N<+\infty)$.
 If
%For a total-constant vector $c^*$,
%let
%${\mathcal H}$ be the set of indices defined in \eqref{eq:nindexi}.
 for a rate-constant vector $\kappa^*$ \textcolor{black}{and a total-constant vector $c^*$}, $G$ has $2N+1$ nondegenerate positive steady states in  ${\mathcal P}_{c^*}$,  then \textcolor{black}{$N+1$ of these positive steady states are stable if the condition \eqref{eq:bwd'} holds, and $N$ of these positive steady states are stable if the condition \eqref{eq:bwd'} does not hold.}
\end{corollary}
\begin{proof}
Notice that if for $(\kappa^*, c^*)$, all positive steady states of $G$ are nondegenerate, then the equality \eqref{eq:lmbdm} in the proof of Lemma \ref{lm:bdm} can be replaced with
\begin{align*}
{\tt sign}(\pd{q}{x_1}(\kappa^*;x_1^{(i)})){\tt sign}(\pd{q}{x_1}(\kappa^*;x_1^{(i-1)}))<0.
\end{align*}
Then, the conclusion follows from a similar argument with the proof of Lemma \ref{lm:bdm}.
\end{proof}

\textcolor{black}{
{\bf Proof of Theorem \ref{thm:bdm}.}}
\begin{proof}
(a) %The conclusion follows from Theorem \ref{thm:nondegmultistable}.
If $cap_{pos}(G)=2N$ $(0\leq N< +\infty)$, then by Theorem \ref{thm:nc}, we have $cap_{nondeg}(G)=2N$. Thus, there exists a choice of $(\kappa^*, c^*)$ such that $G$ has exactly $2N$ nondegenerate positive steady states in ${\mathcal P}_{c^*}$ (note here, there is no degenerate positive steady states since $cap_{pos}(G)=2N$). By Lemma \ref{lm:stable}, Lemma \ref{lm:jh}, and Theorem \ref{thm:sign},
$N$ of these $2N$ positive steady states are stable. Note that by Corollary \ref{cry:sign-stable} (b) and the fact that $cap_{pos}(G)=2N$, we have $cap_{stab}(G)\leq N$. Therefore, we conclude $cap_{stab}(G)= N$.

(b) Assume $cap_{pos}(G)=2N+1$ $(0\leq N<+\infty)$.
Notice that  by Corollary \ref{cry:sign-stable} (b) and by the fact that $cap_{pos}(G)=2N+1$,
we have
\begin{align}\label{eq:proofbdm}
cap_{stab}(G)\leq N+1.
\end{align}
 We have the following two cases.
\begin{itemize}
\item[] ({\bf Case 1.}) If ${\mathcal W}\cap{\mathcal B}\neq \emptyset$, then
by Corollary \ref{cry:ww'}, ${\mathcal W}_{nondeg}\cap{\mathcal B}\neq \emptyset$ (recall the definition of ${\mathcal W}_{nondeg}$ in \eqref{eq:witness'}).
That means there exists a choice of parameters $(\kappa^*, c^*)$ such that
$G$ has $2N+1$ nondegenerate positive steady states and the condition \eqref{eq:bwd'} holds (recall that the inequality stated in the definition of ${\mathcal B}$ \eqref{eq:bwd} is exactly the condition \eqref{eq:bwd'}).
By Corollary \ref{cry:lmbdm},  there are $N+1$ stable positive steady states. So, by \eqref{eq:proofbdm}, we have $cap_{stab}(G)= N+1$.
%\textcolor{red}{the conclusion follows from Corollary \ref{cry:lmbdm}, Corollary \ref{cry:ww'} and Theorem \ref{thm:nondegmultistable}}.
\item[] ({\bf Case 2.})
 Suppose ${\mathcal W}\cap{\mathcal B}= \emptyset$.
 By \eqref{eq:proofbdm}, it is sufficient to prove that $cap_{stab}(G)\leq N$. For any $(\kappa^*,c^*)$, if $(\kappa^*,c^*)\notin{\mathcal W}$, then $G$ has at most $2N$ positive steady states. By Corollary \ref{cry:sign-stable} (b), at most $N$ positive steady states are stable. If $(\kappa^*,c^*)\in{\mathcal W}$, then $G$ has $2N+1$ positive steady states.
 Note that $(\kappa^*,c^*)\notin{\mathcal B}$ since ${\mathcal W}\cap{\mathcal B}= \emptyset$. By Corollary \ref{cry:sign-stable} (b) and Lemma \ref{lm:bdm}, at most $N$ positive steady states are stable. Therefore, we conclude that $cap_{stab}(G)\leq N$.
\end{itemize}
\end{proof}

\textcolor{black}{
{\bf Proof of Corollary \ref{cry:bdm}.}}
\begin{proof}
If $cap_{pos}(G)=2N$ $(0\leq N< +\infty)$, then the conclusion follows from the proof of the part (a) in Theorem \ref{thm:bdm}.
% by Theorem \ref{thm:nc}, $cap_{nondeg}(G)=2N$, and by
%Theorem \ref{thm:nondegmultistable}, $cap_{stab}(G)=N$. Since $cap_{nondeg}(G)=2N$, there exists a choice of $(\kappa^*, c^*)$ such that $G$ has exactly $2N$ nondegenerate positive steady states in ${\mathcal P}_{c^*}$ (note here, there is no degenerate positive steady states since $cap_{pos}(G)=2N$). By Lemma \ref{lm:stable}, Lemma \ref{lm:jh}, and Theorem \ref{thm:sign},
%$N$ of these $2N$ positive steady states are stable.

If $cap_{pos}(G)=2N+1$ $(0\leq N<+\infty)$, we have the following two cases.
\begin{itemize}
\item[] ({\bf Case 1.}) If ${\mathcal W}\cap{\mathcal B}\neq \emptyset$, then
the conclusion follows from {\bf Case 1} in the  proof of Theorem \ref{thm:bdm}.
%\textcolor{red}{the conclusion follows from Corollary \ref{cry:lmbdm}, Corollary \ref{cry:ww'} and Theorem \ref{thm:nondegmultistable}}.
\item[] ({\bf Case 2.})
 If ${\mathcal W}\cap{\mathcal B}= \emptyset$, then by Theorem \ref{thm:bdm} (b), $cap_{stab}(G)=N$.
 Note that by Theorem \ref{thm:nc}, $cap_{nondeg}(G)=2N+1$, and hence,
 there exists a choice of parameters $(\kappa^*, c^*)$ such that $G$ has exactly $2N+1$ nondegenerate positive steady states in ${\mathcal P}_{c^*}$. Since ${\mathcal W}\cap{\mathcal B}= \emptyset$, we know the condition \eqref{eq:bwd'} does not hold for $(\kappa^*, c^*)$ (again, recall the the inequality stated in the definition of ${\mathcal B}$ \eqref{eq:bwd} is the condition \eqref{eq:bwd'}). By Corollary \ref{cry:lmbdm}, $N$ of these positive steady states are stable.
 \end{itemize}
\end{proof}
 \begin{lemma}\label{lm:bd}
 \textcolor{black}{
 Given a network $G$ \eqref{eq:network} with two reactions,
  suppose $cap_{pos}(G)=2N+1$ $(0\leq N<+\infty)$. If
 for a rate-constant vector $\kappa^*$ and a total-constant vector $c^*$, $G$ has $2N+1$ positive
 steady states in ${\mathcal P}_{c^*}$, and if $N+1$ of these steady states are stable, then
 %If $cap_{pos}(G)=3$, and if
 %for a rate-constant vector $\kappa^*$ and a total-constant vector $c^*$, $G$ has three nondegenerate positive
 %steady states, then two of these steady states are stable if and only if
 \begin{align}\label{eq:bd'}
 (\gamma_{\tau1}-\gamma_{\tau2})(\beta_{\tau1}-\alpha_{\tau1})<0,
 \end{align}
 where $\tau$ and $\gamma_{\tau j}$ $(j\in \{1,2\})$ are  defined in \eqref{eq:setk} and \eqref{eq:gamma} for the total constant-vector $c^*$. Especially, if $\tau$ is the only index in $[\tau]$, then
 the condition \eqref{eq:bd'} becomes \eqref{eq:sbd}, i.e.,
 %\begin{align*}%\label{eq:sbd'}
 $(\alpha_{\tau1}-\alpha_{\tau2})(\beta_{\tau1}-\alpha_{\tau1})<0$.
 %\end{align*}
 }\end{lemma}
 \begin{proof}
 By \cite[Lemma 4.1]{Joshi:Shiu:Multistationary}, if  $G$ contains only two reactions (i.e., $m=2$), and if $cap_{pos}(G)\geq 1$, then the stoichiometric subspace is one-dimensional.
 Since $G$ contains  two reactions,  the set of indices ${\mathcal  L}$ defined in \eqref{eq:L} has only one element. So, the left-hand side of the inequality \eqref{eq:bwd'} is
 \begin{align}\label{eq:proofga}
 %g(a)\;=\;
 \begin{cases}
 \left(\beta_{11}-\alpha_{11}\right)
\kappa^*_1{\mathcal C}_1 \prod \limits_{k\in {\mathcal H}, k\neq \tau}(\frac{B_k}{|A_k|}-\frac{A_k}{|A_k|}\frac{B_{\tau}}{A_{\tau}})^{\gamma_{k1}}, & \text{if}\; \gamma_{\tau1}<\gamma_{\tau2} \\
\left(\beta_{12}-\alpha_{12}\right)
\kappa^*_2{\mathcal C}_2 \prod \limits_{k\in {\mathcal H}, k\neq \tau}(\frac{B_k}{|A_k|}-\frac{A_k}{|A_k|}\frac{B_{\tau}}{A_{\tau}})^{\gamma_{k2}}, & \text{if}\; \gamma_{\tau1}>\gamma_{\tau2}
%-\left(\beta_{11}-\alpha_{11}\right)
%\lambda\kappa_2\prod\limits_{j\neq k} (-A_j\frac{B_k}{A_k}+B_j)^{\tilde \alpha_{j2}}, & \text{if}\; \alpha_{k1}>\alpha_{k2}
\end{cases}.
\end{align}
%For any $k\in {\mathcal H}\cap {\mathcal J}_1$ such that $k\neq \tau$,  we show
%below that $B_k-A_k\frac{B_\tau}{A_\tau}>0$. In fact, we have the following cases.
%\begin{itemize}
%\item If $A_k>0$, then by the definition of $k$, we have  $-\frac{B_\tau}{A_\tau}>-\frac{B_k}{A_k}$, and so,
%$$-A_k\frac{B_\tau}{A_\tau}+B_k\;>\;-A_k\frac{B_k}{A_k}+B_k\;=\;0.$$
%\item If $A_j=0$, then $B_j=x_j^{(1)}=x_j^{(2)}=x_j^{(3)}>0$.
%\item If $A_k<0$, then by \eqref{eq:tinterval}, we have  $-\frac{B_\tau}{A_\tau}<-\frac{B_k}{A_k}$, and so,
%$$-A_k\frac{B_\tau}{A_\tau}+B_k\;>\;-A_k\frac{B_k}{A_k}+B_k\;=\;0.$$
%\end{itemize}
By Lemma \ref{lm:well-defined} (2), for any  $k\in {\mathcal H}\backslash \{\tau\}$,
we have $Y_k(-\frac{B_{\tau}}{A_{\tau}})=\frac{1}{|A_k|}(B_k-A_k\frac{B_\tau}{A_\tau})>0$.
By \cite[Lemma 4.1]{Joshi:Shiu:Multistationary}, we have $(\beta_{11}-\alpha_{11})(\beta_{12}-\alpha_{12})<0$.
By Lemma \ref{rmk:mathcalc}, we have ${\mathcal C}_j>0$ for any $j\in \{1,2\}$.
Therefore, the sign of \eqref{eq:proofga} is equal to
$-{\tt sign}(\gamma_{\tau1}-\gamma_{\tau2})(\beta_{11}-\alpha_{11})$.
%Note also, by \cite[Lemma 4.1]{Joshi:Shiu:Multistationary}, we have ${\tt sign}(\beta_{11}-\alpha_{11})=-{\tt sign}(\beta_{12}-\alpha_{12})$.
Note also,
by \eqref{eq:a} and \eqref{eq:setk}, we know $A_{\tau}=\frac{\beta_{\tau1}-\alpha_{\tau1}}{\beta_{11}-\alpha_{11}}>0$.
So,  the inequality \eqref{eq:bwd'} stated in  Lemma \ref{lm:bdm} becomes
$(\gamma_{\tau1}-\gamma_{\tau2})(\beta_{\tau1}-\alpha_{\tau1})<0$ (i.e., the condition \eqref{eq:bd'}).

If there is only one index in the equivalence class $[\tau]$, then by
\eqref{eq:gamma}, we have
$(\gamma_{\tau1}-\gamma_{\tau2})=(\alpha_{\tau1}-\alpha_{\tau2})$. So,
the inequality in \eqref{eq:bd}
(i.e., the condition \eqref{eq:bd'}) becomes \eqref{eq:sbd}.
 \end{proof}

\textcolor{black}{
{\bf Proof of Corollary \ref{cry:bd}.}}
\begin{proof}
  By lemma \ref{lm:bd}, the argument is similar to the proof of Theorem \ref{thm:bdm}.
\end{proof}

\section{Discussion}\label{sec:dis}
Given a network $G$ with a one-dimensional stoichiometric subspace,
   suppose $cap_{pos}(G)=3$. We assume
 for a rate-constant vector $\kappa^*$, $G$ has $3$ nondegenerate positive
 steady states in a stoichiometric compatibility class ${\mathcal P}_{c^*}$.
 Let $q(\kappa; x_1)$ be the function defined in \eqref{eq:q} for the total-constant vector $c^*$.
From Remark \ref{rmk:nobd} and the proof of Theorem \ref{thm:bdm}, we see that the sign of
the left-hand side of the inequality in \eqref{eq:bwd} is the Brouwer degree of the function $q(\kappa^*; x_1)$
over a bounded subset of the interval $I$ defined in \eqref{eq:tinterval}.
By \eqref{eq:bwd}, we see this Brouwer degree depends on the choice of the total-constant vector
$c$ and the rate-constant vector $\kappa$.
We remark that it is proved in \cite[Theorem 3]{CFMW}, if
a dissipative network  admits no boundary steady states, then
the Brouwer degree does not depend on the choice of
$c$ and $\kappa$.
So, a natural question is for a general network $G$,  if there exist two choices
of parameters $c$ and $\kappa$ such that their corresponding Brouwer degrees have different signs but
 the network has
three nondegenerate positive steady states for both choices of parameters. To put it simply, we wonder if there exist two choices of parameters yielding three nondegenerate positive steady states such that one yields multistability but the other one does not.

%\bibliographystyle{siamplain}
%\bibliography{references}
%\newpage

\bigskip

\begin{center}
{\large\bf SUPPLEMENTARY MATERIAL}
\end{center}

Table \ref{tab:sup} lists all files at the online repository:
\url{https://github.com/ZhishuoCode/small-network}

%\textcolor{red}{Hao Xu}

\begin{table}[h!]
\centering
\caption{Supporting Information Files}
\label{tab:sup}
{\scriptsize
\begin{tabular}{||c c c||}
 \hline
\cellcolor[HTML]{FFCE93} Name &\cellcolor[HTML]{FFCE93} File Type & \cellcolor[HTML]{FFCE93}Results \\ [0.5ex]
 \hline\hline
 \texttt{Example.mw/.pdf} & \texttt{Maple/PDF} & Example \ref{ex:counter}/Example \ref{ex:counter2} \\
  \texttt{Plot.mw/.pdf} & \texttt{Maple/PDF} & Example \ref{ex:evenodd}/Figure \ref{fig:main1} \\
 %\texttt{SM.pdf}     & \texttt{PDF}   & Descartes' rule and  Theorem \ref{thm:g2}\\
 %\texttt{witness2.mw/.pdf} & \texttt{Maple/PDF} &  Theorem \ref{thm:g1}\\
% \texttt{reverse4.mw/.pdf} & \texttt{Maple/PDF} & Lemma \ref{lm:g1} \\
 %\texttt{reverse3.mw/.pdf} & \texttt{Maple/PDF}& Lemma \ref{lm:g1} \\
 %\texttt{reverse2.mw/.pdf} & \texttt{Maple/PDF} & Lemma \ref{lm:g1} \\
 %\texttt{reverse1.mw/.pdf} & \texttt{Maple/PDF} & Lemma \ref{lm:g1} \\
 %\texttt{reverse5.mw} & \texttt{Maple} & Lemma \ref{lm:g1} \\
 %\texttt{rrc.mw} & \texttt{Maple} &  Theorem \ref{thm:g1}\\
%\texttt{witness1.mw/.pdf} & \texttt{Maple/PDF} &  Theorem \ref{thm:4-reactant}\\
%\texttt{irreversible.mw/.pdf} & \texttt{Maple/PDF} &  Lemma \ref{lm:4-reactant}\\
 \hline
\end{tabular}
}
\end{table}

\end{document}